\documentclass[journal]{IEEEtran}

\ifCLASSINFOpdf
\usepackage[cmex10]{amsmath}
\usepackage{subfigure}
\usepackage{graphics}
\usepackage{epsfig}

\usepackage{color}
\usepackage{amsfonts}
\usepackage{amssymb}
\usepackage{ntheorem}
\newtheorem{assm}{{Assumption}}
\newtheorem{clam}{{Claim}}
\newtheorem{lema}{{Lemma}}
\newtheorem{thme}{{Theorem}}
\newtheorem{rem}{{Remark}}
\newenvironment*{proof}{{\it Proof.}}{\hfill $\square$\par}
\newtheorem{prope}{\textbf{Proposition}}

\newtheorem{coro}{{Corollary}}
\begin{document}

\title{Event-Triggered Adaptive Control
 of a Parabolic PDE-ODE Cascade with Piecewise-Constant Inputs and Identification}

\author{Ji~Wang,~\IEEEmembership{Member,~IEEE},
       and Miroslav~Krstic,~\IEEEmembership{Fellow,~IEEE}
\thanks{
%This work was supported  by
%the National Science Foundation under Grants 1935329 and 1823983.

J. Wang is with Department of Automation, Xiamen University, Xiamen, Fujian 361005, China (e-mail:jiwang9024@gmail.com).

M. Krstic is with Department of Mechanical and Aerospace Engineering, University of California, San Diego, La Jolla, CA 92093-0411, USA (e-mail: krstic@ucsd.edu).
}}

\maketitle

\begin{abstract}
We present an adaptive event-triggered boundary
control scheme
for a parabolic PDE-ODE system, where the reaction coefficient
of the parabolic PDE, and the system parameter of a scalar ODE, are unknown. In the proposed controller,
the parameter estimates, which are built by batch least-square identification,  are recomputed and the plant states  are resampled simultaneously. As a result, both
the parameter estimates and the control input employ
piecewise-constant values. In the closed-loop system, the following
results are proved: 1) the absence of a Zeno phenomenon; 2) finite-time exact
identification of the unknown parameters under most initial conditions of the plant (all initial conditions except a set of measure zero); 3) exponential regulation of the plant
states to zero.  A
simulation example is presented to validate the theoretical result.
\end{abstract}

\begin{IEEEkeywords}
Parabolic PDEs, adaptive control, backstepping, event-triggered control, least-squares identifier.
\end{IEEEkeywords}

%===============================================================================

\section{Introduction}
\subsection{Boundary control of parabolic PDEs}
Parabolic partial differential equations (PDEs) are predominately used in describing fluid, thermal, and chemical dynamics, including many applications of sea ice melting and freezing \cite{koga2016Backstepping,Wettlaufer1991Heat}, continuous
casting of steel \cite{Petrus2012Enthalpy} and lithium-ion batteries \cite{koga2017State,Tang2017State}.
These therefore give rise to related important control and estimation problems of parabolic PDEs.

The backstepping approach has been verified
as a very powerful tool for boundary stabilization of PDEs. The first backstepping boundary control design for parabolic
PDEs was introduced in \cite{Liu2003Boundary,Liu2000Backstepping}. Subsequently, more results about boundary control of parabolic PDEs have emerged in the past decade, such as
\cite{2015Boundary,Deutscher2015backstepping,Deutscher2016backstepping,Meurer2009Tracking,Orlov2017Output,Pisano2012Boundary}.

In addition to the aforementioned works on parabolic PDEs, topics concerning parabolic PDE-ODE coupled systems are
also popular, which have rich physical background such as coupled
electromagnetic, coupled mechanical, and coupled chemical
reactions \cite{tang2011state}. Backstepping stabilization of a parabolic PDE in cascade with a linear
ODE has been primarily presented in \cite{krstic2009Compensating} with Dirichlet
type boundary interconnection and, the results on Neuman
boundary interconnection were presented in \cite{Susto2000,Tang2011}. Besides, backstepping boundary control designs of a parabolic PDE sandwiched by two ODEs were presented in \cite{Deutscher2020Output,J2019Output}.
\subsection{Event-triggered control of PDEs}
When implementing the continuous-in-time controllers
into digital platforms, sampling needs to be addressed properly to ensure the stability of the closed-loop system. Compared with periodical sampling, i.e., sampled-data scheme \cite{Karafyllis2017}, the event-triggered
strategy is more efficient in the aspect of
using communication and computational resources, because the sampling happens only when needed.

Most current event-triggered control are designed for ODE systems, such as in \cite{Girard2015,Heemels2012,Seuret2014,Tabuada2007}. For event-triggered control in PDE systems, some pioneering attempts  were proposed relaying on distributed (in-domain) control inputs \cite{Selivanov2015,Yao2013} or modal decomposition \cite{Katz2019,Katz2020}. Recently, using the infinite dimensional
control approach, the state-feedback event-triggered boundary control for a class of $2\times 2$ hyperbolic PDEs was first presented in \cite{Espitia2018}, and then was extended to the output-feedback case in \cite{Espitia2020}. Output-feedback event-trigged boundary control of a sandwich hyperbolic PDE system was also proposed in \cite{JiEvent2021}. For parabolic PDEs, the first result on event-triggered boundary control was proposed in \cite{EspitiaArc}, and subsequently, observer-based output-feedback event-triggered boundary control design was presented in  \cite{RathnayakeObserver}. The above-mentioned event-triggered control design only focused on a PDE plant with completely known parameters while there always exist some plant parameters not known exactly in practice, which creates a need for incorporating adaptive technology.
\subsection{Triggered adaptive control of PDEs}
Three traditional adaptive schemes are the Lyapunov design, the passivity-based design, and the swapping design \cite{Krstic2006}, which were initially developed for ODEs in \cite{krstic1995}, and extended to parabolic PDEs  in \cite{Ahmed2016Adaptive,Ahmed2017Adaptive,KrsticSmyshlyaev2008Adaptive,Li2020,Smyshlyaev2007Adaptivea,Smyshlyaev2007Adaptiveb}, and hyperbolic PDEs in \cite{Anfinsen2019Adaptive,Bernard2014}.

Recently, a new adaptive scheme, using a regulation-triggered batch least-square identifier (BaLSI), was introduced in \cite{Karafyllis2019Adaptive,Karafyllis2018Adaptive}, which has at least two significant advantages over the previous traditional adaptive approaches: guaranteeing exponential regulation of the states to zero, as well as finite-time convergence of the estimates to the true values.  {An application of this new adaptive scheme to a two-link manipulator, which is modeled by a highly nonlinear ODE system and subject to four parametric uncertainties, was shown in \cite{Bagheri2021}.} Regarding PDEs, this method has been applied in adaptive control of a parabolic PDE in \cite{Karafyllis2019Adaptive1}, and of first-order hyperbolic PDEs in \cite{JiAdaptive2021,JiACC2021}. In the
above designs, the triggering is employed for the parameter estimator (update law), rather than the control law, where the plant states are not
sampled. Conversely, in \cite{JiAdaptive2020}, an event-triggered adaptive control design
was proposed by employing triggering for the control law
instead of the parameter estimator, where only asymptotic
convergence is achieved.  In this paper, the triggering is
employed for updating both the parameter estimator and the plant states
in the control law. Both the parameter estimates and
the control input thus employ piecewise-constant values, and
exponential regulation of the plant states is achieved.
\subsection{Contributions}
\begin{itemize}
\item As compared to the previous results in \cite{EspitiaArc,RathnayakeObserver} on event-triggered
backstepping control for parabolic PDEs with completely known plant parameters, uncertain plant parameters and additional ODE dynamics are considered in this paper.

\item Different from adaptive control designs of parabolic PDEs in \cite{KrsticSmyshlyaev2008Adaptive,Smyshlyaev2007Adaptivea,Smyshlyaev2007Adaptiveb} where the control input is continuous-in-time, the control input is piecewise-constant in this paper. Moreover, the exponential regulation to zero of the
plant states and exact parameter identification (for all initial conditions of the plant except a set of measure zero) are achieved in this paper.

\item
Compared with the triggered-type adaptive control designs for parabolic PDE in \cite{Karafyllis2019Adaptive1}, where  triggering is only employed for the parameter update
law rather than the plant states in the  controller, in this paper,  triggering is employed for both
updating the parameter estimator and resampling the plant states in the controller and, as a result, both
the parameter estimates and the control input employ
piecewise-constant values. Moreover, an additional uncertain ODE at the uncontrolled boundary of the parabolic PDE is considered in this paper.

\item To the best of our knowledge, this is the first adaptive event-triggered
boundary control result of parabolic PDEs. Moreover, both the control input and the
parameter estimates employ piecewise-constant values. The result is new even if remove the ODE dynamics.
\end{itemize}
\subsection{Organization}
The problem formulation is shown in Section \ref{13sec:problem}. The nominal continuous-in-time control design is presented in Section \ref{13sec:nominalcontrol}. The design of event-triggered control with piecewise-constant parameter identification is proposed in Section \ref{13sec:adaptive}. The absence of a Zeno phenomenon, parameter convergence,  and exponential regulation are proved in
 Section \ref{13sec:sta}.  The effectiveness of the proposed design
is illustrated by an numerical example in Section \ref{13sec:sim}. The conclusion and
future work are presented in Section \ref{sec:conclusion}.
\subsection{Notation} We adopt the following notation.
\begin{itemize}
\item The symbol $\mathbb Z^+$ denotes the set of all non-negative integers, and $\mathbb R_+:=[0,+\infty)$, and $\mathbb R_-:=(-\infty,0)$.

\item Let $U\subseteq \mathbb R^n$ be a set with non-empty interior and let $\Omega\subseteq\mathbb R$
be a set. By $C^0 (U;\Omega)$, we denote the class of continuous
mappings on $U$, which take values in $\Omega$. By $C^k (U;\Omega)$, where
$k \ge 1$, we denote the class of continuous functions on $U$,
which have continuous derivatives of order $k$ on $U$ and take
values in $\Omega$.

\item We use the notation $\mathbb N$ for the set $\{1,2,\cdots\}$, i.e., the natural numbers without 0.

\item We use the notation $L^2(0, 1)$ for the standard space of the equivalence
class of square-integrable, measurable functions defined
on $(0, 1)$ and $\|f\|=\left(\int_0^1 f(x)^2 dx\right)^{\frac{1}{2}}<+\infty$ for $f \in L^2(0, 1)$.

\item For an $I\subseteq \mathbb R_+$, the space $C^0(I;L^2(0,1))$ is the space of continuous mappings $I\ni t\mapsto u[t]\in L^2(0,1)$.

\item Let $u: \mathbb R_+\times [0,1]\rightarrow \mathbb R$  be given. We use the notation $u[t]$ to denote the profile of $u$ at certain $t\ge 0$, i.e., $(u[t])(x)=u(x,t)$, for all $x\in[0,1]$.
\end{itemize}
\section{Problem Formulation}\label{13sec:problem}
We conduct the control design based on the following parabolic PDE-ODE system,
\begin{align}
\dot \zeta (t)& = a\zeta (t) + {b}u(0,t) ,\label{eq:ode1}\\
{u_t}(x,t) &= \varepsilon{u_{xx}}(x,t)+{\lambda}u(x,t),\label{eq:ode2}\\
u_x(0,t)  &= 0,\label{eq:ode4}\\
u_x(1,t)+q u(1,t) &=  U(t),\label{eq:ode5}
\end{align}
for $x\in[0,1],t\in[0,\infty)$, where $\zeta(t)$ is a scalar  ODE state and ${u}(x,t)$ is a PDE state. The function $U(t)$ is the control input to be designed. The parameters $\varepsilon,\lambda, a$ are arbitrary, and $b\neq 0$. The ODE system parameter $a$ and the coefficient $\lambda$ of the PDE in-domain couplings are unknown.

The motivation of the considered PDE-ODE cascade is from control of uncertain thermal-fluid systems with uncertain finite-dimensional sensor dynamics. We only consider a scalar ODE in this paper with the purpose of presenting the proposed control
design  more clearly. With a modest modification, the result in this paper is possible to extend to the case that the ODE state is a vector and the system matrix and input matrix  in the ODE are linear functions of unknown parameters. We reiterate here that the result in this paper is new even if the ODE \eqref{eq:ode1} is absent.
\begin{assm}\label{as:coe1}
The bounds of the unknown parameters $\lambda$, $a$ are known and arbitrary, i.e.,
\begin{align}
&\underline \lambda\le \lambda\le \overline \lambda,\\
&\underline a\le a\le \overline a,
\end{align}
where $\underline \lambda,\overline \lambda,\overline a,\underline a$ are arbitrary positive constants whose values are known.
\end{assm}
{We define a vector $\theta$ which incudes the two unknown parameters $\lambda$, $a$, i.e.,
\begin{align*}
\theta=[\lambda,a]^T.\label{eq:theta}
\end{align*}}\begin{assm}\label{eq:As1}
The parameter $q$ satisfies
\begin{align}
q>\frac{1}{4}+\frac{\overline\lambda}{2\varepsilon}.
\end{align}
\end{assm}
Assumption \ref{eq:As1} avoids the use of the signal $u(1,t)$ in the nominal control law, whose purpose is to ensure no Zeno behavior in the event-based system. It should be mentioned
that an eigenfunction expansion of the solution of \eqref{eq:ode2}--\eqref{eq:ode5} with $U(t)=0$ shows that the parabolic PDE system is unstable when
$\lambda>\varepsilon\pi^2/4$, no matter what $q>0$ is \cite{RathnayakeObserver}.
\section{Nominal control design}\label{13sec:nominalcontrol}
In the nominal control design, similar to \cite{Deutscher2020Output}, we apply three transformations to convert the original plant \eqref{eq:ode1}--\eqref{eq:ode5} to an exponentially stable target system.

We introduce the following backstepping transformation ((4.55) in \cite{krstic2008}),
\begin{align}
w (x,t) = u(x,t) - \int_0^x {\Psi}(x,y)u(y,t)dy,\label{eq:contran1b}
\end{align}
where
\begin{align}
{\Psi}(x,y)=-\frac{\lambda}{\varepsilon}x\frac{I_1(\sqrt{\lambda(x^2-y^2)/\varepsilon})}{\sqrt{\lambda(x^2-y^2)/\varepsilon}},\label{eq:Psi}
\end{align}
and $I_1$ denotes the modified Bessel functions of the
first kind,
to convert the plant \eqref{eq:ode1}--\eqref{eq:ode5} to the  system
\begin{align}
\dot \zeta (t)& = a\zeta (t) + {b}w(0,t) ,\label{eq:I1}\\
{w_t}(x,t) &= \varepsilon{w_{xx}}(x,t),\label{eq:I2}\\
w_x(0,t)  &= 0,\label{eq:I3}\\
w_x(1,t)+r w(1,t) &=U(t)+(r-q-{\Psi}(1,1)) u(1,t)\notag\\&\quad - \int_0^1 \bigg({\Psi}_x(1,y)+q{\Psi}(1,y)\notag\\&\quad+{\Psi}(1,y)(r-q)\bigg)u(y,t)dy,\label{eq:I4}
\end{align}
where $r$ is a positive constant which will be determined later.

Following Section 4.5 in \cite{krstic2008}, the inverse transformation of \eqref{eq:contran1b} is shown as
\begin{align}
u (x,t) = w(x,t) - \int_0^x {\Phi}(x,y)w(y,t)dy,\label{eq:contran1bI}
\end{align}
where
\begin{align}
{\Phi}(x,y)=-\frac{\lambda}{\varepsilon}x\frac{J_1(\sqrt{\lambda(x^2-y^2)/\varepsilon})}{\sqrt{\lambda(x^2-y^2)/\varepsilon}},\label{eq:Phi}
\end{align}
and $J_1$ denotes the (nonmodified) Bessel functions of the
first kind.

Applying the second transformation,
\begin{align}
v(x,t) = w(x,t) - \gamma(x)\zeta(t),\label{eq:contran2b}
\end{align}
where
\begin{align}
\gamma(x)=-\kappa\cos\left(\sqrt{\frac{(b\kappa-a)}{\varepsilon}}x\right),\label{eq:ga}
\end{align}
we convert the system \eqref{eq:I1}--\eqref{eq:I4} into the following intermediate system,
\begin{align}
\dot \zeta (t) =& -a_{\rm m}\zeta (t) + {b}v (0,t) ,\label{eq:Ib1}\\
{v _t}(x,t)=& {\varepsilon}{v _{xx}}(x,t)-\gamma(x){b}v (0,t),\label{eq:Ib2}\\
v_x(0,t)=&0\label{eq:Ib3}\\
v_x(1,t)+rv(1,t)=&U(t)- \gamma'(1)\zeta(t)- r\gamma(1)\zeta(t)\notag\\& +(r-q-{\Psi}(1,1)) u(1,t)\notag\\& - \int_0^1 \bigg({\Psi}_x(1,y)+q{\Psi}(1,y)\notag\\&+{\Psi}(1,y)(r-q)\bigg)u(y,t)dy, \label{eq:Ib4}
\end{align}
where
\begin{align}
a_{\rm m}=b\kappa-a>0\label{eq:am}
\end{align}
is ensured by a design parameter $\kappa$ satisfying
\begin{align}
\kappa>\frac{\overline a}{b}.\label{eq:kapa}
\end{align}
Please see Appendix-A for the calculation of the conditions of $\gamma(x)$ by matching \eqref{eq:I1}--\eqref{eq:I4} and \eqref{eq:Ib1}--\eqref{eq:Ib4} via \eqref{eq:contran2b}.

We introduce the third transformation
\begin{align}
\beta(x,t)=v(x,t)-\int_0^x h(x,y)v(y,t)dy\label{eq:thirdtran}
\end{align}
where the kernel $h(x,y)$ is to be determined,
to convert \eqref{eq:Ib1}--\eqref{eq:Ib4} into the target system
\begin{align}
\dot \zeta (t) =& -a_{\rm m}\zeta (t) + {b}\beta (0,t) ,\label{eq:targ1}\\
{\beta _t}(x,t)=& {\varepsilon}{\beta _{xx}}(x,t),\label{eq:targ2}\\
\beta_x(0,t)=&0,\label{eq:targ3}\\
\beta_x(1,t)=&-r\beta(1,t), \label{eq:targ4}
\end{align}
where
\begin{align}
r=q+{\Psi}(1,1)+h(1,1)=q-\frac{\lambda}{2\varepsilon}>\frac{1}{4}\label{eq:r}
\end{align}
with recalling ${\Psi}(1,1)=-\frac{\lambda}{2\varepsilon}$ considering \eqref{eq:Psi}, and $h(1,1)=0$ derived from the last two conditions of $h(x,y)$ which are shown next, and Assumption \ref{eq:As1}.

By matching \eqref{eq:targ1}--\eqref{eq:targ3} and \eqref{eq:Ib1}--\eqref{eq:Ib3} via \eqref{eq:thirdtran} (see Appendix-B for details), the conditions of $h(x,y)$ are obtained as
\begin{align}
&h_y(x,0){\varepsilon}+\gamma(x){b}-{b}\int_0^x h(x,y)\gamma(y)dy=0,\label{eq:h1}\\
&h_{yy}(x,y)-h_{xx}(x,y)=0,\label{eq:h2}\\
&h_y(x,x)+h_x(x,x)=0,\label{eq:h3}\\
&h(0,0)=0,\label{eq:h4}
\end{align}
which is covered by the ones found in \cite{Deutscher2018backstepping} which ensures that \eqref{eq:h1}--\eqref{eq:h4} have a piecewise $C_2$-solution.

For \eqref{eq:targ4} to hold, the control input is chosen as
\begin{align}
U(t)=&\int_0^1 K_1(1,y;\theta)u(y,t)dy+K_2(1;\theta)\zeta(t),\label{eq:U}
\end{align}
where
\begin{align}
K_1(1,y;\theta)=&{\Psi}_x(1,y)+h_x(1,y)+rh(1,y)\notag\\&+\left(q-\frac{\lambda}{2\varepsilon}\right){\Psi}(1,y)\notag\\&-\int_y^1 (h_x(1,z)+rh(1,z)) {\Psi}(z,y)dz,\\
K_2(1;\theta)=&\gamma'(1)+ \left(q-\frac{\lambda}{2\varepsilon}\right)\gamma(1)\notag\\&-\int_0^1 (h_x(1,y)+rh(1,y))\gamma(y)dy.\label{eq:K2}
\end{align}
Writing $\theta=[\lambda,a]^T$ in $K_1$, $K_2$ emphasizes the fact that  $K_1$, $K_2$ depend on the unknown parameters $\lambda$, $a$ (the kernels $\Psi,\gamma, h$ defined in \eqref{eq:Psi}, \eqref{eq:ga}, \eqref{eq:h1}--\eqref{eq:h4} include these unknown parameters).

According to \cite{Deutscher2020Output}, there exists kernel $h^I(x,y)\in\mathbb R$ for the inverse transformation of \eqref{eq:thirdtran}, which is shown as
\begin{align}
v(x,t)=\beta(x,t)-\int_0^x h^I(x,y)\beta(y,t)dy.\label{eq:thirdtranI}
\end{align}
\section{Event-triggered control design with piecewise-constant parameter identification}\label{13sec:adaptive}
Based on the  nominal continuous-in-time feedback \eqref{eq:U}, we give the form of
an adaptive event-triggered control law $U_d$, as follows:
\begin{align}
U_{di}:= U_d(t_i)=&  \int_0^1 K_1(1,y;\hat\theta(t_i))u(y,t_i)dy\notag\\&+K_2 (1;\hat\theta(t_i))\zeta(t_i)\label{eq:dU}
\end{align}
for $t\in[t_i,t_{i+1})$, where
\begin{align}
\hat\theta=[\hat \lambda, \hat a]^T\label{eq:hattheta}
\end{align}
is an estimate, which is generated with a triggered batch least-squares identifier (BaLSI), of the two unknown parameters $\lambda,a$. The identifier, and the sequence of time instants $\{t_i\ge 0\}_{i=0}^{\infty}$, $i\in \mathbb Z^+$ are defined in the next subsection.

Inserting the piecewise-constant control input $U_{di}$ into \eqref{eq:ode5}, the boundary condition becomes
\begin{align}
u_x(1,t)+qu(1,t) =U_{di}.\label{eq:(1)}
\end{align}

When we mention the continuous-in-state control signal $U_c$, we refer to the control input consisting of triggered parameter estimates and continuous states, i.e.,
\begin{align}
& U_c(t) = \int_0^1 K_1(1,y;\hat\theta(t_i))u(y,t)dy+K_2 (1;\hat\theta(t_i))\zeta(t)\label{eq:cU}
\end{align}
for $t\in[t_i,t_{i+1})$.
Define the difference between the continuous-in-state control signal $U_c$ in \eqref{eq:cU} and the event-triggered control input $U_d$ in \eqref{eq:dU} as $d(t)$, given by
\begin{align}
d(t)&=U_c(t)-U_d(t)\notag\\
&=\int_0^1 K_1(1,y;\hat\theta(t_i))(u(y,t)-u(y,t_i))dy\notag\\
&\quad+K_2 (1;\hat\theta(t_i))(\zeta(t)-\zeta(t_i))\label{eq:d}
\end{align}
for $t\in[t_i,t_{i+1})$, which reflects the deviation of the plant states from their sampled values. The signal $U_c$ dose not act as the control input of the plant but used in the ETM (i.e., $d(t)$) which will be shown latter.

Define the difference between the continuous-in-state control signal $U_c(t)$ in \eqref{eq:cU} and the nominal continuous-in-time control input $U(t)$ in \eqref{eq:U} as $p(t)$, given by
\begin{align}
p(t)=&U(t)-U_c(t)\notag\notag\\
=&\int_0^1 (K_1(1,y;\theta)-K_1(1,y;\hat\theta(t_i)))u(y,t)dy\notag\\
\quad&+(K_2 (1;\theta)-K_2 (1;\hat \theta(t_i)))\zeta(t),~~t\in[t_i,t_{i+1})\label{eq:p}
\end{align}
which reflects the deviation of the estimates from the actual unknown parameters.

The deviations $d(t)$ and $p(t)$ will be used in the following design and analysis.
\subsection{Event-Triggering Mechanism}
The sequence of time instants $\{t_i\ge 0\}_{i=0}^{\infty}$ ($t_0=0$) is defined as
\begin{align}
{t_{i+1}} = \min\{\inf \{ t > {t_{i}}:d{(t)^2} \ge  -\xi m(t)\}, t_i+T\},\label{eq:tk1}
\end{align}
where the positive constant $\xi$ is a design parameter, and another design parameter $T>0$ sets the maximum dwell-time with the purpose of avoiding the less
frequent updates of the parameter estimates which may lead to a large overshoot in the
response of the closed-loop system.

The dynamic variable $m(t)$ in \eqref{eq:tk1} satisfies the ordinary differential equation,
\begin{align}
\dot m(t) = & - \eta m(t)+\lambda_d d(t)^2- {\kappa _1}{u}{(1,t)^2} - {\kappa _2}u{(0,t)^2}\notag\\
&-{\kappa _3}\|u(\cdot,t)\|^2-{\kappa _4}\zeta(t)^2\label{eq:dm}
\end{align}
for $t\in(t_i,t_{i+1})$ with  $m(t_0)=m(0)<0$, and $m(t_i^-)=m(t_i)=m(t_i^+)$. Here, $t_i^+$ and $t_i^-$
 are the right
and left limits of $t=t_i$. It is worth noting that the initial condition
for $m(t)$ in each time interval has been chosen such that $m(t)$
is time-continuous. The positive design parameters $\kappa_1,\kappa_2,\kappa_3,\kappa_4$ are determined later. Inserting $d(t)^2\le -\xi m(t)$ guaranteed by \eqref{eq:tk1} into \eqref{eq:dm}, we have
\begin{align}
\dot m(t) \le & - (\eta+\lambda_d\xi) m(t)- {\kappa _1}{u}{(1,t)^2} - {\kappa _2}u{(0,t)^2}\notag\\
&-{\kappa _3}\|u(\cdot,t)\|^2-{\kappa _4}\zeta(t)^2.
\end{align}
Recalling $m(0)<0$ and applying the comparison principle, we have that $m(t)<0$ all the time.
\subsection{Batch Least-Squares Identifier}\label{sec:ls}
According to  \eqref{eq:ode1},   \eqref{eq:ode2}, we get for $\tau>0$ and $ n=1,2,\cdots$ that
\begin{align}
&\frac{d}{d\tau}\left(\int_0^{1}\sin({x\pi n})u(x,\tau)dx-\frac{1}{b}\varepsilon \pi n\zeta(\tau)\right)\notag\\
=&{\varepsilon}\int_0^{1}\sin({x\pi n})u_{xx}(x,\tau)dx\notag\\
&+\lambda\int_0^{1}\sin({x\pi n})u(x,\tau)dx-\frac{1}{b}\varepsilon \pi n\dot\zeta(t)\notag\\
=&-\varepsilon \pi n\cos({\pi n})u(1,\tau)+\varepsilon \pi nu(0,\tau)\notag\\&-\frac{a}{b}\varepsilon \pi n\zeta(t)-\varepsilon \pi nu(0,\tau)\notag\\
&-\varepsilon \pi^2 n^2\int_0^{1}\sin({x\pi n})u(x,\tau)dx+\lambda\int_0^{1}\sin({x\pi n})u(x,\tau)dx\notag\\
=&-\varepsilon \pi n\cos({\pi n})u(1,\tau)-\varepsilon \pi^2 n^2\int_0^{1}\sin({x\pi n})u(x,\tau)dx\notag\\&+\lambda\int_0^{1}\sin({x\pi n})u(x,\tau)dx-\frac{1}{b}a\varepsilon \pi n\zeta(t).\label{eq:Ls1}
\end{align}
Define
\begin{align}
\mu_{i+1}:=\min\{t_d:d\in\{0,\ldots,i\},t_d\ge t_{i+1}-\tilde N T\},\label{13eq:mui}
\end{align}
according to \cite{Karafyllis2018Adaptive}, where the positive integer $\tilde N\ge 1$ is a free design parameter (a lager $\tilde N$ will
make the least-squares identifier run based on a bigger set of
data, which makes the identifier more robust with respect to
measurement errors), and the positive constant $T$ is the maximum dwell-time in \eqref{eq:tk1}.
Integrating \eqref{eq:Ls1} from $\mu_{i+1}$ to $t$, yields
\begin{align}
&f_n(t,\mu_{i+1})=\lambda g_{n,1}(t,\mu_{i+1})+ag_{n,2}(t,\mu_{i+1}),\label{eq:f}
\end{align}
where
\begin{align}
f_n(t,\mu_{i+1})=&\int_0^{1}\sin({x\pi n})u(x,t)dx-\frac{1}{b}\varepsilon \pi n\zeta(t)\notag\\&-\int_0^{1}\sin({x\pi n})u(x,\mu_{i+1})dx+\frac{1}{b}\varepsilon \pi n\zeta(\mu_{i+1})\notag\\
&+\int_{\mu_{i+1}}^t\bigg[\varepsilon \pi n(-1)^nu(1,\tau)\notag\\
&+\varepsilon \pi^2 n^2\int_0^{1}\sin({x\pi n})u(x,\tau)dx\bigg]d\tau,\label{eq:fn}\\
g_ {n,1}(t,\mu_{i+1})=&\int_{\mu_{i+1}}^t\int_0^{1}\sin({x\pi n})u(x,\tau)dxd\tau,\label{eq:gq1}\\
g_ {n,2}(t,\mu_{i+1})=&-\frac{1}{b}\varepsilon \pi n\int_{\mu_{i+1}}^t\zeta(\tau)d\tau,\label{eq:gq2}
\end{align}
for $ n=1,2,\cdots$.

Define the function $h_{i,n}:\mathbb R^{2}\to \mathbb R_+$ by the formula:
\begin{align}
h_{i,n}(\ell)=&\int_{\mu_{i+1}}^{t_{i+1}}\big(f_n(t,\mu_{i+1})-\ell_1 g_{n,1}(t,\mu_{i+1})\notag\\&-\ell_2g_{n,2}(t,\mu_{i+1})\big)^2dt\label{eq:hi}
\end{align}
for $i\in  \mathbb Z^+$, $n\in \mathbb N$,
$\ell=[\ell_1,\ell_2]^T$.

According to \eqref{eq:f}, the function $h_{i,n}(\ell)$ in \eqref{eq:hi} has a global minimum $h_{i,n}(\theta)=0$. We get from Fermat's theorem (vanishing gradient at extrema) that the following matrix equation hold for every $i\in  \mathbb Z^+$ and $n=1,2,\cdots$:
\begin{align}
Z_n(\mu_{i+1},t_{i+1})=G_n(\mu_{i+1},t_{i+1})\theta\label{eq:Fer}
\end{align}
where
\begin{align}
&Z_n(\mu_{i+1},t_{i+1})=[H_{n,1}(\mu_{i+1},t_{i+1}),H_{n,2}(\mu_{i+1},t_{i+1})]^T,\label{eq:bZ}\\
&G_n(\mu_{i+1},t_{i+1})=\left[
    \begin{array}{ccc}
      Q_{n,1}(\mu_{i+1},t_{i+1}) & Q_{n,2}(\mu_{i+1},t_{i+1})\\
     Q_{n,2}(\mu_{i+1},t_{i+1}) & Q_{n,3}(\mu_{i+1},t_{i+1})\\
    \end{array}\label{eq:bG}
  \right]
\end{align}
with
\begin{align}
H_{n,1}(\mu_{i+1},t_{i+1})&=\int_{\mu_{i+1}}^{t_{i+1}}g_{n,1}(t,\mu_{i+1})f_{n}(t,\mu_{i+1}) dt,\label{eq:H1m}\\
H_{n,2}(\mu_{i+1},t_{i+1})&=\int_{\mu_{i+1}}^{t_{i+1}}g_{n,2}(t,\mu_{i+1})f_{n}(t,\mu_{i+1}) dt,\label{eq:H2m}\\
Q_{n,1}(\mu_{i+1},t_{i+1})&=\int_{\mu_{i+1}}^{t_{i+1}}g_{n,1}(t,\mu_{i+1})^2 dt,\label{eq:Q1m}\\
Q_{n,2}(\mu_{i+1},t_{i+1})&=\int_{\mu_{i+1}}^{t_{i+1}}g_{n,1}(t,\mu_{i+1})g_{n,2}(t,\mu_{i+1}) dt,\label{eq:Q2m}\\
Q_{n,3}(\mu_{i+1},t_{i+1})&=\int_{\mu_{i+1}}^{t_{i+1}}g_{n,2}(t,\mu_{i+1})^2 dt.\label{eq:Q3m}
\end{align}

Indeed, \eqref{eq:Fer} is obtained by differentiating the functions $h_{i,n}(\ell)$ defined by \eqref{eq:hi} with respect to $\ell_1$, $\ell_2$, respectively, and evaluating the derivatives at the position of the global minimum $(\ell_1,\ell_2)=(\lambda,a)$.

The parameter estimator (update law) is defined as
\begin{align}
&\hat\theta(t_{i+1})={\rm argmin}\bigg\{|\ell-\hat\theta(t_i)|^2: {\ell\in \Theta}, \notag\\
&Z_n(\mu_{i+1},t_{i+1})=G_n(\mu_{i+1},t_{i+1})\ell,~~ n=1,2,\cdots\bigg\},\label{eq:adaptivelaw}
\end{align}
where
\begin{align}
{\Theta=\{\ell\in \mathbb R^2:\underline \lambda\le \ell_1\le \overline \lambda ,\underline a\le\ell_2\le \overline a\}}.
\end{align}
\subsection{Well-Posedness Issues}
Integrating \eqref{eq:ode1}, we have
\begin{align}
\zeta (t) = a\int_0^t\zeta (\tau)d\tau + {b}\int_0^tu(0,\tau)d \tau+\zeta (0),~~t\ge 0.\label{eq:odef}
\end{align}
The well-posedness property is stated as follows.
{\begin{prope}\label{13Pro:1}
For every $u[t_i]\in L^2(0,1)$, $\zeta(t_i)\in \mathbb R$, $m(t_i)\in \mathbb R_-$, there exist  unique mappings $u\in C^0([t_i,t_{i+1}];L^2(0,1))\cap C^1((t_i,t_{i+1})\times[0,1])$ with $u[t]\in C^2([0,1])$, $\zeta\in C^0([t_i,t_{i+1}];\mathbb R)$, $m\in C^0([t_i,t_{i+1}];\mathbb R_{-})$, which satisfy \eqref{eq:ode2}, \eqref{eq:ode4}, \eqref{eq:dU}, \eqref{eq:(1)}, \eqref{eq:odef}, and \eqref{eq:dm}.
\end{prope}}
\begin{proof}
According to Proposition 4 in \cite{RathnayakeObserver}, whose proof depends on Theorem 4.11 in \cite{Karafyllis2019Input}, and recalling \eqref{eq:dm}, \eqref{eq:odef}, Proposition \ref{13Pro:1} is obtained.
\end{proof}
\begin{figure}
\centering
\includegraphics[width=8cm]{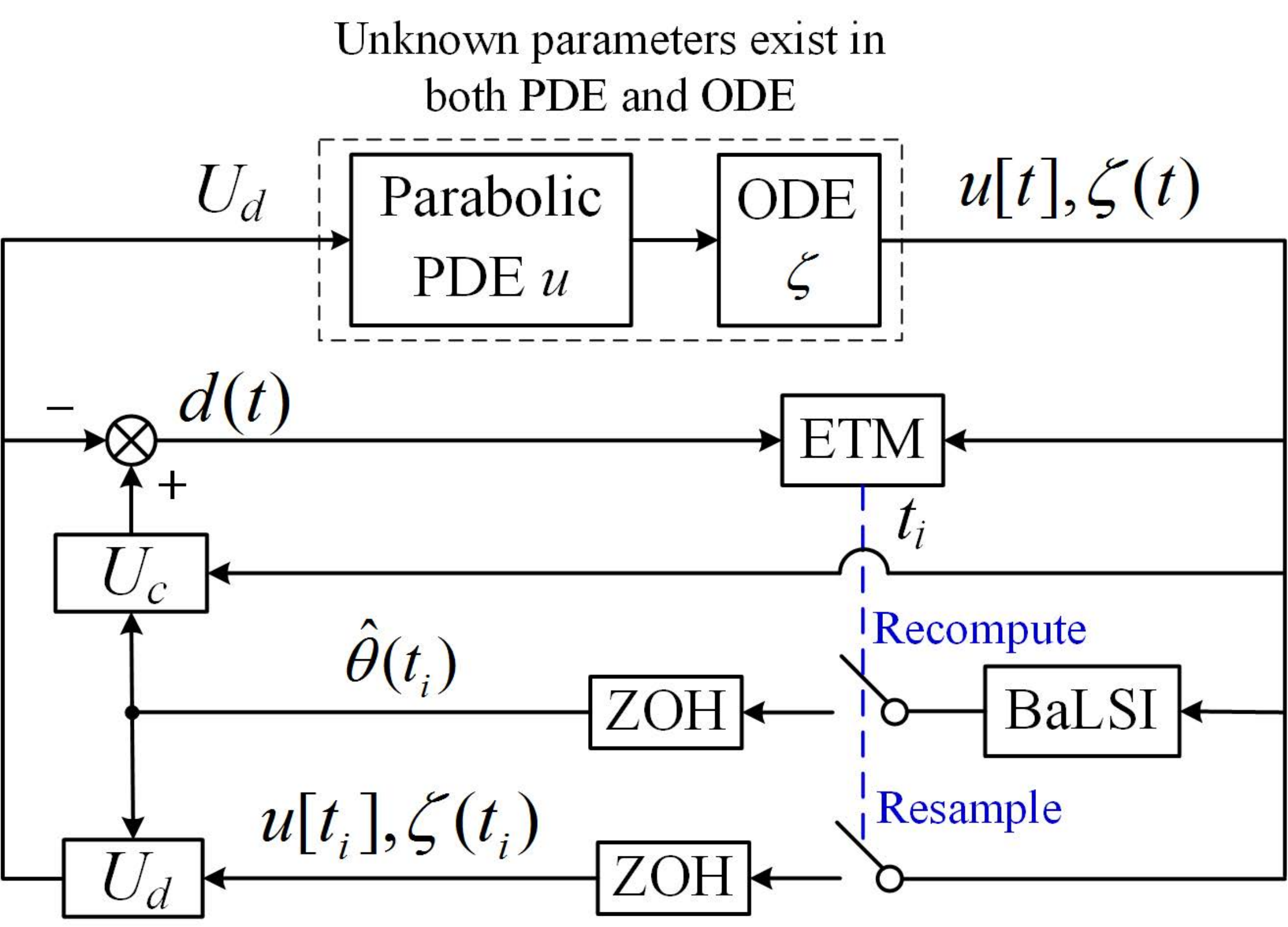}
\caption{Block diagram of the closed-loop system.}
\label{fig:1}
\end{figure}
\section{Main result}\label{13sec:sta}
{The block diagram of the closed-loop system is presented in Figure \ref{fig:1}. Before presenting the main result of this paper, we present the following technical lemmas first.}
\begin{lema}\label{lem:dd}
For $d(t)$ defined in \eqref{eq:d}, there exist positive constants ${\epsilon_1},{\epsilon_2},{\epsilon_3},{\epsilon_4},{\epsilon_5}$ such that
\begin{align}
\dot d(t)^2\le& {\epsilon_1} d(t)^2+{\epsilon_2}u(1,t)^2+{\epsilon_3} u(0,t)^2\notag\\
&+{\epsilon_4}\|u(\cdot,t)\|^2+{\epsilon_5}\zeta(t)^2\label{eq:dd}
\end{align}
for $t\in(t_i,t_{i+1})$, where ${\epsilon_1},{\epsilon_2},{\epsilon_3},{\epsilon_4},{\epsilon_5}$ only depend on the design parameter $\kappa$ in \eqref{eq:kapa}, the known plant parameters, and the known bounds $\overline a,\underline a, \overline \lambda, \underline \lambda$ of the unknown parameters.
\end{lema}
\begin{proof}
The event-triggered control input $U_d$ is constant on $t\in(t_i,t_{i+1})$, i.e., $\dot U_d(t)=0$. Taking the time derivative of \eqref{eq:d}, recalling \eqref{eq:ode1}--\eqref{eq:ode4}, we obtain that
\begin{align}
\dot d(t)=&\dot U_c(t)\notag\\
=&\int_0^1 K_1(1,y;\hat\theta(t_i))u_t(y,t)dy+K_2 (1;\hat\theta(t_i))\dot\zeta(t)\notag\\
=&\int_0^1 K_1(1,y;\hat\theta(t_i))\varepsilon u_{xx}(y,t)dy\notag\\&+\int_0^1 K_1(1,y;\hat\theta(t_i))\lambda u(y,t)dy\notag\\
&+K_2 (1;\hat\theta(t_i))a\zeta(t)+K_2 (1;\hat\theta(t_i))bu(0,t)\notag\\
=&K_1(1,1;\hat\theta(t_i))\varepsilon u_{x}(1,t)-K_{1y}(1,1;\hat\theta(t_i))\varepsilon u(1,t)\notag\\
&+K_{1y}(1,0;\hat\theta(t_i))\varepsilon u(0,t)\notag\\&+\int_0^1 K_{1yy}(1,y;\hat\theta(t_i))\varepsilon u(y,t)dy\notag\\&+\int_0^1 K_1(1,y;\hat\theta(t_i))\lambda u(y,t)dy\notag\\
&+K_2 (1;\hat\theta(t_i))a\zeta(t)+K_2 (1;\hat\theta(t_i))bu(0,t).\label{eq:dd1}
\end{align}
Applying  \eqref{eq:d}, \eqref{eq:p}, it allows us to rewrite \eqref{eq:(1)} as
\begin{align}
u_x(1,t)+qu(1,t)=U(t)-p(t)-d(t).\label{eq:newplant}
\end{align}
Inserting \eqref{eq:newplant} into \eqref{eq:dd1} to replace $u_{x}(1,t)$, we then have
\begin{align}
\dot d(t)=&K_1(1,1;\hat\theta(t_i))\varepsilon (-q u(1,t)+U(t)-p(t)-d(t))\notag\\
&-K_{1y}(1,1;\hat\theta(t_i))\varepsilon u(1,t)\notag\\
&+K_{1y}(1,0;\hat\theta(t_i))\varepsilon u(0,t)\notag\\&+\int_0^1 K_{1yy}(1,y;\hat\theta(t_i))\varepsilon u(y,t)dy\notag\\&+\int_0^1 K_1(1,y;\hat\theta(t_i))\lambda u(y,t)dy\notag\\
&+K_2 (1;\hat\theta(t_i))a\zeta(t)+K_2 (1;\hat\theta(t_i))bu(0,t)\notag\\
=&-K_1(1,1;\hat\theta(t_i))\varepsilon d(t)-K_1(1,1;\hat\theta(t_i))\varepsilon p(t)\notag\\
& -(qK_1(1,1;\hat\theta(t_i))\varepsilon+K_{1y}(1,1;\hat\theta(t_i))\varepsilon) u(1,t)\notag\\
&+(K_{1y}(1,0;\hat\theta(t_i))\varepsilon+K_2 (1;\hat\theta(t_i))b) u(0,t)\notag\\
&+\int_0^1 [K_{1yy}(1,y;\hat\theta(t_i))\varepsilon+K_1(1,y;\hat\theta(t_i))\lambda\notag\\&+K_1(1,1;\hat\theta(t_i))\varepsilon K_1(1,y;\theta)] u(y,t)dy\notag\\
&+(K_2 (1;\hat\theta(t_i))a+K_1(1,1;\hat\theta(t_i))\varepsilon K_2(1;\theta))\zeta(t),\label{eq:dde}
\end{align}
where \eqref{eq:U} has been used.

Applying the Cauchy-Schwarz inequality into \eqref{eq:p}, we have
\begin{align}
p(t)^2\le &2\max_{\vartheta_1,\vartheta_2\in\Theta}\left\{\int_0^1 (K_1(1,y;\vartheta_1)-K_1(1,y;\vartheta_2))^2dy\right\}\|u[t]\|^2\notag\\
\quad&+2\max_{\vartheta_1,\vartheta_2\in\Theta}\{(K_2 (1;\vartheta_1)-K_2 (1;\vartheta_2))^2\}\zeta(t)^2.\label{eq:pineq}
\end{align}
Applying the Cauchy-Schwarz inequality into \eqref{eq:dde}, using \eqref{eq:pineq}, we then obtain \eqref{eq:dd}, where
\begin{align}
{\epsilon_1}=&6\varepsilon^2\max_{\vartheta\in\Theta}\{K_1(1,1;\vartheta)^2\},\label{eq:mu1}\\
{\epsilon_2}=&6\max_{\vartheta\in\Theta}\{(qK_1(1,1;\vartheta)\varepsilon+K_{1y}(1,1;\vartheta)\varepsilon)^2\},\label{eq:mu2}\\
{\epsilon_3}=&6\max_{\vartheta\in\Theta}\{(K_{1y}(1,0;\vartheta)\varepsilon+K_2 (1;\vartheta)b)^2\},\\
{\epsilon_4}=&12\max_{\vartheta\in\Theta}\left\{\int_0^1 (K_{1yy}(1,y;\vartheta)\varepsilon+K_1(1,y;\vartheta)\lambda)^2dy\right\}\notag\\&+12\varepsilon^2\max_{\vartheta\in\Theta}\{K_1(1,1;\vartheta)^2\}\max_{\vartheta\in\Theta}\left\{\int_0^1  K_1(1,y;\vartheta)^2dy\right\}\notag\\&+12\varepsilon^2\max_{\vartheta\in\Theta}\{K_1(1,1;\vartheta)^2\}\notag\\&\quad\times\max_{\vartheta_1,\vartheta_2\in\Theta}\left\{\int_0^1 (K_1(1,y;\vartheta_1)-K_1(1,y;\vartheta_2))^2dy\right\},\\
{\epsilon_5}=&12{\bar a}^2\max_{\vartheta\in\Theta}\{(K_2 (1;\vartheta)^2\}\notag\\
&+12\varepsilon^2\max_{\vartheta\in\Theta}\{K_1(1,1;\vartheta)^2\} \max_{\vartheta\in\Theta}\left\{K_2(1;\vartheta)^2\right\}\notag\\
&+12\varepsilon^2\max_{\vartheta\in\Theta}\{K_1(1,1;\vartheta)^2\}\notag\\&\quad\times\max_{\vartheta_1,\vartheta_2\in\Theta}\{(K_2 (1;\vartheta_1)-K_2 (1;\vartheta_2))^2\}.\label{eq:mu5}
\end{align}
The proof of Lemma \ref{lem:dd} is complete.
\end{proof}
Relying on Lemma \ref{lem:dd}, we present the following lemma which shows that the minimal dwell-time is a positive constant.
\begin{lema}\label{lem:dwell}
For some positive $\kappa_1,\kappa_2,\kappa_3,\kappa_4$, there exists a minimal dwell-time $\underline\tau>0$  such that $t_{i+1}-t_i\ge\underline\tau$ for all $i\in \mathbb Z^+$.
\end{lema}
\begin{proof}
1) If the event is triggered by the second condition, i.e., $t_{i+1}=t_i+T$, in \eqref{eq:tk1}, it is obvious that the dwell-time is $T>0$.

2) Next, we consider the case that the event is triggered by the first condition in \eqref{eq:tk1}.
Define the following  function $\psi(t)$:
\begin{align}
\psi(t)  = \frac{{d{{(t)}^2} + \frac{1}{2}\xi m(t)}}{{- \frac{1}{2}\xi m(t)}},\label{eq:psi}
\end{align}
which was originally introduced in \cite{Espitia2018}. We have that $\psi(t_{i+1}) =1$ because the event is triggered,  and that $\psi(t_{i}) <0$ because of $d(t_i)=0$ according to \eqref{eq:d}.
The function $\psi(t)$ is a continuous function on $[t_i^+,t_{i+1}^-]$ recalling Proposition \ref{13Pro:1} and \eqref{eq:d}. By the intermediate value
theorem, there exists $t^*>t_i$ such that $\psi(t)\in[0,1]$ when $t\in[t^*,t_{i+1}^-]$. The minimal dwell-time can be founded as the minimal time it takes for $\psi(t)$ from 0 to 1.

Taking the derivative of \eqref{eq:psi} for all $t\in[t^*,t_{i+1})$,  applying Young's inequality, using \eqref{eq:dd} in Lemma \ref{lem:dd}, and inserting  \eqref{eq:dm},  we  have
\begin{align}
\dot \psi  = &\frac{{2d(t)\dot d(t) + \frac{1}{2}{\xi}\dot m(t)}}{{- \frac{1}{2}\xi m(t)}} - \frac{{\dot m(t)}}{{  m(t)}}\psi\notag \\
 \le & \frac{1}{{- \frac{1}{2}\xi m(t)}} \bigg[{\epsilon_1}d(t)^2+{\epsilon_2} u(1,t)^2+{\epsilon_3} u(0,t)^2\notag\\&+{\epsilon_4}\|u(\cdot,t)\|^2+{\epsilon_5}\zeta(t)^2+ d{{(t)}^2} - \frac{1}{2}{\xi}\eta m(t)\notag\\
  &+\frac{1}{2}{\xi}\lambda_d d(t)^2- \frac{1}{2}{\xi}{\kappa _1}{u}{(1,t)^2}
   - \frac{1}{2}{\xi}{\kappa _2}u{(0,t)^2}\notag\\&- \frac{1}{2}{\xi}{\kappa _3}\|{u}{(\cdot,t)\|^2}- \frac{1}{2}{\xi}{\kappa _4}{\zeta}{(t)^2}\bigg]-\frac{\lambda_d d(t)^2}{m(t)}\psi+\eta\psi\notag\\
&- \frac{- {\kappa _1}{u}{(1,t)^2} - {\kappa _2}u{(0,t)^2}- {\kappa _3}\|{u}{(\cdot,t)\|^2}- {\kappa _4}{\zeta}{(t)^2}}{  m(t)}\psi.\label{eq:dpsi}
 \end{align}
It is worth pointing out that the last term in \eqref{eq:dpsi} is less than zero.
Choose
\begin{align}
\kappa_1&\ge \frac{2{\epsilon_2}}{{\xi}},\\ \kappa_2&\ge \frac{2{\epsilon_3}}{{\xi}},\\
\kappa_3&\ge \frac{2{\epsilon_4}}{{\xi}},\\ \kappa_4&\ge \frac{2{\epsilon_5}}{{\xi}},
\end{align}
where ${\epsilon_2},{\epsilon_3},{\epsilon_4},{\epsilon_5}$ are given in \eqref{eq:mu2}--\eqref{eq:mu5}, which only depend on the design parameter $\kappa$ in \eqref{eq:kapa}, the known plant parameters, and the known bounds $\overline a,\underline a, \overline \lambda, \underline \lambda$ of the unknown parameters.

Then \eqref{eq:dpsi} becomes
 \begin{align}
 \dot \psi  \le & \frac{1}{{- \frac{1}{2}{\xi} m(t)}} \bigg[\left({\epsilon_1}+1+\frac{1}{2}{\xi}\lambda_d\right) d{{(t)}^2} - \frac{1}{2}{\xi}\eta m(t)\bigg]\notag\\&-\frac{\lambda_d d(t)^2}{m(t)}\psi+\eta\psi.\label{eq:dpsi1}
\end{align}
Inserting
\begin{align}
 \frac{d(t)^2}{{m(t)}} = &\frac{d(t)^2+\frac{1}{2}{\xi}m(t)-\frac{1}{2}{\xi}m(t)}{{m(t)}}\notag\\ =& - \frac{1}{2}{\xi}\left(\psi(t)+1\right),
 \end{align}
 we obtain from \eqref{eq:dpsi1} that
 \begin{align}
 \dot \psi\le  n_1 \psi^2+n_2 \psi+n_3,\label{eq:df}
\end{align}
where
\begin{align}
n_1&=\frac{1}{2}{\lambda_d}{\xi},\\
n_2&=1+{\epsilon_1}+{\xi}\lambda_d+\eta,\\
n_3&=1+\eta+{\epsilon_1}+\frac{1}{2}{\xi}\lambda_d
\end{align}
are positive constants.
It follows that  the lower bound of dwell-time in this case is
\begin{align}
\underline\tau_a=\int_0^1\frac{1}{n_1+n_2s+n_3s^2}ds>0.\label{eq:undertaua}
\end{align}
Together with the result in 1), we have that the minimal dwell-time $\underline\tau$ is
\begin{align}
\underline\tau=\min\{\underline\tau_a,T\}>0.\label{eq:undertau}
\end{align}
The proof of this Lemma is complete.
\end{proof}
It follows from Lemma \ref{lem:dwell} that no Zeno phenomenon occurs, i.e.,
$\lim_{i\to \infty} t_i=+\infty$.
\begin{coro}\label{col}
For all initial data $u[0]\in L^2(0,1)$, $\zeta(0)\in \mathbb R$, $m(0)\in \mathbb R_-$, there exist unique mappings   $u\in C^0(\mathbb R_+;L^2(0,1))\bigcap C^1(J\times[0,1])$ with $u[t]\in C^2([0,1])$, $\zeta\in C^0(\mathbb R_{+};\mathbb R)$, $m\in C^0(\mathbb R_{+};\mathbb R_{-})$, which satisfy \eqref{eq:ode2}, \eqref{eq:ode4}, \eqref{eq:(1)}, \eqref{eq:odef}, with \eqref{eq:dU}, \eqref{eq:dm} for $t>0$, where $J=\mathbb R_+\backslash\{t_i\ge 0, i=0,1,2,\ldots\}$.
\end{coro}
\begin{proof}
It is an immediate consequence of Proposition \ref{13Pro:1} and
Lemma \ref{lem:dwell}. Indeed, the solution is constructed (by the step
method) iteratively between successive triggering times.
\end{proof}
\begin{lema}\label{cl:nsQ0}
The sufficient and necessary conditions of $Q_{n,1}(\mu_{i+1},t_{i+1})=0$, $Q_{n,3}(\mu_{i+1},t_{i+1})=0$ for $n=1,2,\ldots$, are $u[t]= 0$, $\zeta(t)= 0$ on $t\in[\mu_{i+1},t_{i+1}]$, respectively.
\end{lema}
\begin{proof}
Necessity: If $Q_{n,1}(\mu_{i+1},\tau_{i+1})=0$ for $n=1,2,\ldots$, then the definition \eqref{eq:Q1m} in conjunction with continuity of $g_{n,1}(t,\mu_{i+1})$ for $t\in[\mu_{i+1},\tau_{i+1}]$ (a consequence of definition \eqref{eq:gq1} and the fact that $u\in C^0([\mu_{i+1},\tau_{i+1}];L^2(0,1))$) implies
\begin{align}
g_{n,1}(t,\mu_{i+1})=0,~~ t\in[\mu_{i+1},\tau_{i+1}].\label{eq:gq10}
\end{align}
According to the definition \eqref{eq:gq1} and continuity of the mapping $\tau\to \int_0^{1}\sin({x\pi n})z[\tau]dx$ (a consequence of the fact that $u\in C^0([\mu_{i+1},\tau_{i+1}];L^2(0,1))$, \eqref{eq:gq10} implies
\begin{align}
\int_0^{1}\sin({x\pi n})u(x,\tau)dx=0, ~~\tau\in[\mu_{i+1},\tau_{i+1}]
\end{align}
for $n=1,2,\ldots$.
Since the set $\{\sqrt{2}\sin(n\pi x):n=1,2,\ldots\}$ is an orthonormal basis of $L^2(0,1)$, we have $u[t]=0$ for $t\in[\mu_{i+1},\tau_{i+1}]$.

Sufficiency: If $u[t]=0$ on $t\in[\mu_{i+1},\tau_{i+1}]$, then $Q_{n,1}(\mu_{i+1},\tau_{i+1})=0$ for $n=1,2,\ldots$ is obtained directly, according to \eqref{eq:gq1}, \eqref{eq:Q1m}.

By recalling \eqref{eq:gq2}, \eqref{eq:Q3m},  the fact that the sufficient and necessary condition of $Q_{3n}(\mu_{i+1},t_{i+1})=0$ for $n=1,2,\ldots$ is $\zeta(t)=0$ on $t\in[\mu_{i+1},t_{i+1}]$, is obtained straightforwardly.

The proof of Lemma \ref{cl:nsQ0} is complete.\end{proof}
\begin{lema}\label{13lem:theta}
For the adaptive estimates defined by \eqref{eq:adaptivelaw} based on the data in the interval $t\in[\mu_{i+1},t_{i+1}]$, the following statements hold:

If $u[t]$ (or $\zeta(t)$) is not identically zero for $t\in[\mu_{i+1},t_{i+1}]$, then $\hat \lambda(t_{i+1})=\lambda$ (or $\hat a(t_{i+1})=a$).

If $u[t]$ (or $\zeta(t)$) is  identically zero for $t\in[\mu_{i+1},t_{i+1}]$, then $\hat \lambda(t_{i+1})=\hat \lambda(t_{i})$ (or $\hat a(t_{i+1})=\hat a(t_{i})$).
\end{lema}
\begin{proof}
First, we define a set as follows,
\begin{align}
&S_{i}:=\bigg\{\bar \ell=(\ell_1,\ell_2)^T\in \Theta: Z_n(\mu_{i+1},t_{i+1})=G_n(\mu_{i+1},t_{i+1})\bar \ell,\notag\\
&~~~~~~~~~~n=1,2,\ldots,\bigg\},~~i\in\mathbb Z^+,\label{eq:setadaptivelaw1}
\end{align}
where ${\Theta=\{\bar \ell\in \mathbb R^2:\underline \lambda\le \ell_1\le \overline \lambda ,\underline a\le \ell_2\le \overline a\}}$, and $Z_n$, $G_n$ in \eqref{eq:bZ}, \eqref{eq:bG} are associated with the plant states over a time interval $[\mu_{i+1},t_{i+1}]$.

We prove the following four claims, from which the statements in this lemma are concluded.
\begin{clam}\label{cl:lem4a}
If $u[t]$ is not identically zero and $\zeta(t)$ is identically zero on $t\in[\mu_{i+1},t_{i+1}]$, then $\hat \lambda(t_{i+1})=\lambda$, $\hat a(t_{i+1})=\hat a(t_{i})$.
\end{clam}
\begin{proof}
Because $u[t]$ is not identically zero and $\zeta(t)$ is identically zero on $t\in[\mu_{i+1},\tau_{i+1}]$, there exists $n\in{\mathbb N}$ such that $Q_{n,1}(\mu_{i+1},\tau_{i+1})\neq0$ recalling Lemma \ref{cl:nsQ0}. Define the index set $I$ to be the set of all $n\in{\mathbb N}$ with $Q_{n,1}(\mu_{i+1},\tau_{i+1})\neq0$. According to \eqref{eq:gq2} and $\zeta(t)$ being identically zero on $t\in[\mu_{i+1},\tau_{i+1}]$, we know that $g_{n,2}(t,\mu_{i+1})=0$ on $t\in[\mu_{i+1},\tau_{i+1}]$ for all $n\in \mathbb N$. It follows that $Q_{n,2}(\mu_{i+1},\tau_{i+1})=0$, $Q_{n,3}(\mu_{i+1},\tau_{i+1})=0$, $H_{n,2}(\mu_{i+1},\tau_{i+1})=0$ for all $n\in \mathbb N$ recalling \eqref{eq:Q2m}, \eqref{eq:Q3m} and \eqref{eq:H2m}.
Recalling \eqref{eq:bZ}, \eqref{eq:bG}, then \eqref{eq:setadaptivelaw1} implies  $S_i=\{(\ell_1,\ell_2)\in \Theta:\ell_1=\frac{H_{n,1}(\mu_{i+1},\tau_{i+1})}{Q_{n,1}(\mu_{i+1},\tau_{i+1})}, n\in I\}$. Because $(q_1,q_2)\in S_i$ according to \eqref{eq:Fer}, it follows that $S_i=\{(q_1,\ell_2)\in \Theta:\underline a\le\ell_2\le \overline a\}$. Therefore, \eqref{eq:adaptivelaw} shows that $\hat \lambda(\tau_{i+1})=\lambda$ and $\hat a(\tau_{i+1})=\hat a(\tau_{i})$.
\end{proof}
\begin{clam} If $u[t]$ is  identically zero and $\zeta(t)$ is not identically zero on $t\in[\mu_{i+1},t_{i+1}]$, then  $\hat \lambda(t_{i+1})=\hat \lambda(t_{i})$, $\hat a(t_{i+1})=a$.
\end{clam}
\begin{proof}
The proof of this claim is very similar to the proof of Claim \ref{cl:lem4a}, and thus it is omitted.
\end{proof}
\begin{clam}
If $u[t]$, $\zeta(t)$ are identically zero on $t\in[\mu_{i+1},t_{i+1}]$, then $\hat \lambda(t_{i+1})=\hat \lambda(t_{i})$, $\hat a(t_{i+1})=\hat a(t_{i})$.
\end{clam}
\begin{proof}
In this case, $Q_{n,1}(\mu_{i+1},t_{i+1})=0$,  $Q_{n,2}(\mu_{i+1},t_{i+1})=0$, $Q_{n,3}(\mu_{i+1},t_{i+1})=0$, $H_{n,1}(\mu_{i+1},t_{i+1})=0$, $H_{n,2}(\mu_{i+1},t_{i+1})=0$ for all $n\in \mathbb N$ according to \eqref{eq:gq1}, \eqref{eq:gq2}, \eqref{eq:H1m}--\eqref{eq:Q3m}.
It follows that {$S_i=\Theta$}, and then \eqref{eq:adaptivelaw} shows that $\hat \lambda(t_{i+1})=\hat \lambda(t_{i})$, $\hat a(t_{i+1})=\hat a(t_{i})$.
\end{proof}
\begin{clam}\label{cl:lem4d}
If both $u[t]$ and $\zeta(t)$ are not identically zero on $t\in[\mu_{i+1},t_{i+1}]$, then $\hat \lambda(t_{i+1})=\lambda$, $\hat a(t_{i+1})=a$.
\end{clam}

\begin{proof}
By virtue of \eqref{eq:Fer}, \eqref{eq:adaptivelaw}, if $S_i$ is a singleton then it is nothing else but the least-squares estimate of the unknown vector of parameters $(\lambda,a)$ on the interval $[\mu_{i+1},t_{i+1}]$, and $S_i=\{(\lambda,a)\}$.
From \eqref{eq:bZ}, \eqref{eq:bG}, \eqref{eq:setadaptivelaw1}, we have that
\begin{align}
&S_i\subseteq S_{ai}:=\bigg\{(\ell_1,\ell_2)\in \Theta: \ell_2=\frac{H_{n,2}(\mu_{i+1},t_{i+1})}{Q_{n,3}(\mu_{i+1},t_{i+1})}\notag\\
&-\ell_1\frac{Q_{n,2}(\mu_{i+1},t_{i+1})}{Q_{n,3}(\mu_{i+1},t_{i+1})}, n=1,2,\cdots\bigg\}.\label{13eq:S2}
\end{align}
We next prove by contradiction that $S_i=\{(\lambda,a)\}$.
Suppose that on the contrary $S_i\neq \{(\lambda,a)\}$, i.e., $S_i$ defined by \eqref{eq:setadaptivelaw1} is not a singleton, which implies the set $S_{ai}$ defined by \eqref{13eq:S2} are not singletons (because $S_{ai}$ being a singleton implies that $S_i$ is a singleton). It follows that there exist constants $\bar r\in\mathbb R$ such that
\begin{align}
\frac{Q_{n,2}(\mu_{i+1},t_{i+1})}{Q_{n,3}(\mu_{i+1},t_{i+1})}=\bar r, ~n\in \mathbb N,\label{13eq:Q1}
 \end{align}
because if there were two different indices $k_1,k_2\in \mathbb N$ with $\frac{Q_{k_1,2}(\mu_{i+1},t_{i+1})}{Q_{k_1,3}(\mu_{i+1},t_{i+1})}\neq \frac{Q_{k_2,2}(\mu_{i+1},t_{i+1})}{Q_{k_2,3}(\mu_{i+1},t_{i+1})}$, then the set $S_{ai}$ defined by \eqref{13eq:S2} would be a singleton.

Moreover, since $S_i$ is not a singleton, the definition \eqref{eq:setadaptivelaw1} implies
\begin{align}
Q_{n,2}(\mu_{i+1},t_{i+1})^2=Q_{n,1}(\mu_{i+1},t_{i+1})Q_{n,3}(\mu_{i+1},t_{i+1})\label{13eq:Q2}
\end{align}
for all $n\in \mathbb N$ by recalling \eqref{eq:bG}. According to \eqref{eq:Q1m}--\eqref{eq:Q3m}, and the fact that the Cauchy-Schwarz inequality holds as  equality only when two functions are linearly dependent,  we obtain the existence of constants $\check\mu_{n}$ such that
\begin{align}
&g_{{n},1}(t,\mu_{i+1})=\check\mu_{n} g_{{n},2}(t,\mu_{i+1}), ~~n\in \mathbb N\label{13eq:qm2}
\end{align}
for $t\in[\mu_{i+1},t_{i+1}]$ ($g_{{n},2}(t,\mu_{i+1})$ are not identically zero on $t\in[\mu_{i+1},t_{i+1}]$).

Recalling \eqref{13eq:Q1}, we obtain from \eqref{eq:Q1m}--\eqref{eq:Q3m} and \eqref{13eq:qm2} that
\begin{align}
g_{n,1}(t,\mu_{i+1})&=\overline\mu g_{n,2}(t,\mu_{i+1}),~~ \overline\mu\neq 0,~~n\in \mathbb N\label{13eq:bl}
\end{align}
for $t\in[\mu_{i+1},t_{i+1}]$. The reason of the constant $\overline\mu\neq 0$ is given as follows. According to Lemma \ref{cl:nsQ0}, there exists $n_1\in{\mathbb N}$ such that ${Q_{n_1,1}(\mu_{i+1},t_{i+1})}\neq 0$. Hence, $g_{n_1,1}(t,\mu_{i+1})$ is not identically zero on $[\mu_{i+1},t_{i+1}]$.

Equations \eqref{13eq:bl} holding is a necessary condition of the hypothesis that $S_i$ is not a singleton.
Recalling \eqref{eq:gq1}, \eqref{eq:gq2}, and Proposition \ref{13Pro:1}, the fact that the equation \eqref{13eq:bl} holds implies
\begin{align}
\int_0^{1}\sin({x\pi n})u(x,t)dx+\frac{1}{b}\overline\mu\varepsilon \pi n\zeta(t)=0,\label{eq:c}
\end{align}
for $t\in(\mu_{i+1},t_{i+1})$, $x\in[0,1]$, $n\in \mathbb N$.

Taking the time derivative of \eqref{eq:c}, we have that
\begin{align}
&\int_0^{1}\sin({x\pi n})u_t(x,t)dx+\frac{1}{b}\overline\mu\varepsilon \pi n(a\zeta(t)+bu(0,t))\notag\\
=&\int_0^{1}\sin({x\pi n})(\varepsilon u_{xx}(x,t)+\lambda u(x,t))dx\notag\\&+\frac{1}{b}\overline\mu\varepsilon \pi n(a\zeta(t)+bu(0,t))\notag\\
=&-(-1)^n\pi n\varepsilon u(1,t)+\pi n\varepsilon u(0,t)\notag\\&-\int_0^{1}\pi^2 n^2\sin({x\pi n})\varepsilon u(x,t)dx\notag\\&+\int_0^{1}\sin({x\pi n})\lambda u(x,t)dx\notag\\&+\frac{1}{b}\overline\mu\varepsilon \pi na\zeta(t)+\overline\mu\varepsilon \pi nu(0,t)\notag\\
=&-(-1)^n\pi n\varepsilon u(1,t)+n(\pi \varepsilon+\overline\mu\varepsilon \pi ) u(0,t)\notag\\&-\frac{1}{b}\overline\mu\varepsilon \pi n(\lambda-a-\pi^2 n^2\varepsilon) \zeta(t)\notag\\
=&-(-1)^n\pi \varepsilon u(1,t)+(\pi \varepsilon+\overline\mu\varepsilon \pi ) u(0,t)\notag\\&-\frac{1}{b}\overline\mu\varepsilon \pi (\lambda-a-\pi^2 n^2\varepsilon) \zeta(t)=0\label{eq:pr}
\end{align}
for $t\in(\mu_{i+1},t_{i+1})$, $n\in\mathbb N$, where \eqref{eq:c} is applied in going from the second equation to the third one in \eqref{eq:pr}.
Considering any two odd (or even) positive integers $n_1\neq n_2$, we obtain from \eqref{eq:pr} that
$(n_1^2-n_2^2) \zeta(t)=0$ for $t\in(\mu_{i+1},t_{i+1})$.
Considering the fact that $\zeta\in C^0([t_i,t_{i+1}];\mathbb R)$ and $\zeta(t)$ is not identically zero on $t\in[t_i,t_{i+1}]$,  one obtains
$n_1^2=n_2^2$: contradiction.
Consequently,  $S_i$ is  a singleton, i.e., $S_i=\{(\lambda,a)\}$. Therefore, $\hat \lambda(t_{i+1})=\lambda, \hat a(t_{i+1})=a$.
\end{proof}
From Claims \ref{cl:lem4a}--\ref{cl:lem4d}, we obtain Lemma \ref{13lem:theta}.
\end{proof}
\begin{lema}\label{lem:keep}
If $\hat \lambda(t_{i})=\lambda$ (or $\hat a(t_{i})=a$) for certain $i\in  \mathbb Z^+$, then $\hat \lambda(t)=\lambda$ (or $\hat a(t)=a$) for all $t\in[t_i,+\infty)$.
\end{lema}
\begin{proof}
According to Lemma \ref{13lem:theta}, we have that $\hat \lambda(t_{i+1})$ is equal to either $\lambda$ or $\hat \lambda(t_i)$. Therefore, if  $\hat \lambda(t_i)=\lambda$, then $\hat \lambda(t_{i+1})=\lambda$. Repeating the above process, then $\hat \lambda(t)=\lambda$ for all $t\in[t_i,\lim_{k\to \infty}(t_k))$. Recalling Lemma \ref{lem:dwell} which implies $\lim_{k\to \infty}(t_k)\to\infty$, we thus have $\hat \lambda(t)=\lambda$ for $t\in[t_i,\infty)$. The same is true of $\hat a$. The proof is complete.
\end{proof}
\begin{lema}\label{lem:K2}
If $u[0]=0,\zeta(0)\neq 0$, and the user-selected initial estimates $\hat\lambda(0),\hat a(0)$ happen to make $K_2(1;\hat\theta(0))=K_2(1;\hat\lambda(0),\hat a(0))=0$, the constant $K_2(1;\hat\theta(t_1))\neq 0$ is ensured just by changing $\hat\lambda(0)$ as another value (arbitrary) in $[\underline\lambda,\overline\lambda]$.
\end{lema}
\begin{proof}
Because the kernels $h$ and $\gamma$ in $K_2$ \eqref{eq:K2} only include the unknown parameter: $a$, considering $\hat a(t_1)=a$ ensured by $\zeta(0)\neq 0$ with Lemma \ref{13lem:theta}, and $\hat\lambda(t_1)=\hat\lambda(0)$ due to the fact that $u[t]$ is identically zero on $t\in[0,t_1]$ (which is the result of $K_2(1;\hat\theta(0))=0$ with \eqref{eq:ode1}--\eqref{eq:ode4}, \eqref{eq:dU}, \eqref{eq:(1)} and $u[0]=0$), we have that
\begin{align}
K_2(1;\hat\theta(t_1))&=K_2(1;\hat\lambda(t_1),\hat a(t_1))=K_2(1;\hat\lambda(0),a)\notag\\&=K_2(1;\theta)+\frac{\lambda-\hat\lambda(0)}{2\varepsilon}\gamma(1).
\end{align}
If $K_2(1;\hat\theta(t_1))=0$, it implies that
\begin{align}
\hat\lambda(0)=\lambda+\frac{2\varepsilon K_2(1;\theta)}{\gamma(1)}.
\end{align}
Therefore, once we pick another $\hat\lambda(0)$, then $K_2(1;\hat\theta(t_1))\neq 0$ is ensured in the situation mentioned in this lemma.

The proof of Lemma \ref{lem:K2} is complete.
\end{proof}
\begin{rem}\label{eq:re1}
\emph{If $u[0]=0,\zeta(0)\neq 0$, and $K_2(1;\hat\theta(0))=0$ is found under the user-selected initial estimates $\hat\theta(0)=[\hat\lambda(0),\hat a(0)]^T$, then $\hat\lambda(0)$ should be changed as another value (arbitrary) in $[\underline\lambda,\overline\lambda]$.
}
\end{rem}
According to Lemma \ref{lem:K2}, the purpose of Remark \ref{eq:re1} is to avoid the appearance of an extreme case that $u[0]=0$, $\zeta(0)\neq 0$, $K_2(1;\hat\theta(0))=0$, $K_2(1;\hat\theta(t_1))=0$, which implies $K_2(1;\hat\theta(t))=K_2(1;\hat\theta(t_1))=K_2(1;\hat\lambda(0),a)=0$ for $t\ge t_1$, and leads to that the regulation on the ODE dynamics \eqref{eq:ode1} is lost, i.e., $u[t]\equiv0$ for all time while $\zeta(t)$ dynamics may be unstable, according to \eqref{eq:ode1}--\eqref{eq:ode4}, \eqref{eq:dU}, \eqref{eq:(1)}.

Let $\hat\theta(0)$ belong to $\Theta_1$ which is equal to $\Theta$ under Remark \ref{eq:re1}, we obtain the following parameter convergence property.
\begin{lema}\label{lem:convergence}
For $\hat\theta(0)\in \Theta_1$, and all $u[0]\in L^2(0,1)$, $\zeta(0)\in \mathbb R$ except for the case that both $u[0]$ and $\zeta(0)$ are zero, we have
\begin{align}
\hat\lambda (t)=\lambda,~~~~\hat a (t)=a \label{eq:convergence}
\end{align}
for all $t\ge t_2$.
\end{lema}
\begin{proof}
Case 1: $u(x,0)$ is not identically zero for $x\in[0,1]$, and $\zeta(0)$ is not zero. We know that $u[t],\zeta(t)$ are not identically zero on $t\in[0,t_1]$. Recalling Lemmas \ref{13lem:theta}, \ref{lem:keep}, we obtain   \eqref{eq:convergence}.

Case 2: $u[0]=0$ and $\zeta(0)\neq 0$.

We know that $\zeta(t)$ is not identically zero on $t\in[0,t_1]$. If $K_2(1;\hat\lambda(0),\hat a(0))\neq0$, we have that $u[t]$ is not identically zero on $t\in[0,t_1]$ according to \eqref{eq:ode1}--\eqref{eq:ode4}, \eqref{eq:dU}, \eqref{eq:(1)}. Then it is straightforward to obtain \eqref{eq:convergence} with recalling Lemmas \ref{13lem:theta}, \ref{lem:keep}.

If $K_2(1;\hat\lambda(0),\hat a(0))=0$, then $u[t]=0$ on $t\in[0,t_1]$ according to \eqref{eq:ode1}--\eqref{eq:ode4}, \eqref{eq:dU}, \eqref{eq:(1)} and $u[0]=0$. It follows that $\zeta(t)=\zeta(0)e^{a t}$ in \eqref{eq:ode1}, which is not zero on $t\in[0,t_1]$ under $\zeta(0)\neq 0$. Recalling Lemma \ref{lem:K2} and Remark \ref{eq:re1}, we have $K_2(1;\hat\theta(t_1))\neq 0$, which results in that $u[t]$ is not identically zero on $t\in[t_1,t_2]$ considering \eqref{eq:ode1}--\eqref{eq:ode4}, \eqref{eq:dU}, \eqref{eq:(1)}, and the fact that $\zeta(t_1)=\zeta(0)e^{a t_1}$ is not zero. Therefore, we obtain \eqref{eq:convergence} from Lemmas \ref{13lem:theta}, \ref{lem:keep}.

Case 3: $u(x,0)$ is not identically zero, and $\zeta(0)= 0$.

It is obvious that $u[t]$ is not identically zero on $t\in[0,t_1]$. Supposing that $\zeta[t]$ is identically zero on $t\in[0,t_1]$, it follows from \eqref{eq:ode1} that $u(0,t)=0$  on $t\in[0,t_1]$. Applying the method of separation of variables shown in (3.4)--(3.10) in \cite{krstic2008}, it implies from \eqref{eq:ode2}, \eqref{eq:ode4} and $u(0,t)=0$ that $u[t]$ is identically zero on $t\in[0,t_1]$: contradiction. Therefore, $\zeta[t]$ is also not identically zero on $t\in[0,t_1]$. We thus obtain \eqref{eq:convergence} from Lemmas \ref{13lem:theta}, \ref{lem:keep}.

The proof of this lemma is complete.
\end{proof}
{We are now ready to show the main result of this paper. }
\begin{thme}\label{13th:part1}
For all initial data $u[0]\in L^2(0,1)$, $\zeta(0)\in \mathbb R$, $m(0)\in \mathbb R_-$, and $\hat\theta(0)\in \Theta_1$, the closed-loop system, i.e., \eqref{eq:ode1}--\eqref{eq:ode5} under the controller \eqref{eq:dU}, with the event-triggering mechanism \eqref{eq:tk1}, \eqref{eq:dm}, and the least-squares identifier defined by \eqref{eq:adaptivelaw}, has the following properties:

1) Except for the case that both $u[0]$ and $\zeta(0)$ are zero, there exist positive constants $M,\sigma$ (independent of initial conditions) such that
\begin{align}
\Omega(t)\le M\Omega(0)e^{-\sigma t},~~t\in [0,\infty)\label{13eq:th1norm1}
\end{align}
where
\begin{align}
\Omega(t)=\|u[t]\|^2+ \zeta(t)^2+|m(t)|+\left|\tilde\theta(t)\right|.
\end{align}
The signal bars $|\cdot|$ for $\tilde\theta(t)=\theta-\hat\theta(t)$ denotes the Euclidean norm.

2) If both $u[0]$ and $\zeta(0)$ are zero, all signals are bounded in the sense of
\begin{align}
\Omega(t)\le |m(0)|+\left|\theta-\hat\theta(0)\right|,~~t\in [0,\infty).\label{eq:bound}
\end{align}
\end{thme}
\begin{proof}
1) Now we prove the first of the two portions of the theorem.

Define a Lyapunov function as
\begin{align}
V(t) = \frac{1}{2}r_a\int_0^1 \beta(x,t)^2 dx+\frac{1}{2}r_c\zeta(t)^2-m(t)\label{eq:Va}
\end{align}
where $m(t)$ is defined in \eqref{eq:dm}.

Defining
\begin{align}
\bar\Omega(t)=\|\beta[t]\|^2+ \zeta(t)^2+|m(t)|,\label{eq:bom}
\end{align}
we have
\begin{align}
\xi_3\bar\Omega(t)\le V(t)\le \xi_4\bar\Omega(t)\label{eq:barOmega}
\end{align}
where
\begin{align}
\xi_3=&\min\left\{\frac{1}{2}r_a,\frac{1}{2}r_c,1\right\}>0,\notag\\
\xi_4=&\max\left\{\frac{1}{2}r_a,\frac{1}{2}r_c,1\right\}>0.
\end{align}
According to \eqref{eq:newplant} and the nominal control design in Section \ref{13sec:nominalcontrol}, in the triggered control system, the right boundary condition of the target system \eqref{eq:targ1}--\eqref{eq:targ4} becomes
\begin{align}
\beta_x(1,t)+r\beta(1,t)= -p(t)-d(t).\label{eq:1targ5}
\end{align}
For $t\in(t_i,t_{i+1})$, $i\in \mathbb Z^+$, taking the derivative of \eqref{eq:Va} along \eqref{eq:targ1}--\eqref{eq:targ3}, \eqref{eq:1targ5},  recalling \eqref{eq:dm},
we have that
\begin{align}
\dot V(t) =&r_a\int_0^1 \beta(x,t)\beta_t(x,t) dx+r_c\zeta(t)\dot\zeta(t)-\dot m(t)\notag\\
&+r_a\int_0^1 \beta(x,t)\varepsilon \beta_{xx}(x,t) dx+r_c\zeta(t)\dot\zeta(t)-\dot m(t)\notag\\
 =& r_a\varepsilon \beta(1,t) \beta_{x}(1,t)-r_a\varepsilon \beta(0,t) \beta_{x}(0,t)\notag\\&-r_a\varepsilon\int_0^1 \beta_{x}(x,t)^2dx
 -r_ca_{\rm m}\zeta(t)^2\notag\\&+r_c\zeta(t)b\beta(0,t)+ \eta m(t)-\lambda_d d(t)^2  \notag\\
&+{\kappa _1}{\beta}{(1,t)^2} +{\kappa _2}{\beta(0,t)^2}+{\kappa _3}\|u(\cdot,t)\|^2+{\kappa _4}\zeta(t)^2\notag\\
 =& r_a\varepsilon \beta(1,t) (-r\beta(1,t)-p(t)-d(t))\notag\\
 &-r_a\varepsilon\int_0^1 \beta_{x}(x,t)^2dx-a_{\rm m}r_c\zeta(t)^2+r_c\zeta(t)b\beta(0,t)\notag\\
 &+ \eta m(t)-\lambda_d d(t)^2+ {\kappa _1}{u}{(1,t)^2} + {\kappa _2}{u(0,t)^2}\notag\\&+{\kappa _3}\|u(\cdot,t)\|^2+ {\kappa _4}\zeta(t)^2.\label{eq:Veq}
 \end{align}
 Recalling \eqref{eq:contran1bI}, \eqref{eq:contran2b}, \eqref{eq:thirdtranI}, we have
 \begin{align}
 &u (x,t) \notag\\
 = &\beta(x,t)-\int_0^x h^I(x,y)\beta(y,t)dy + \gamma(x)\zeta(t)\notag\\
  &- \int_0^x {\Phi}(x,y)\left(\beta(y,t)-\int_0^y h^I(y,z)\beta(z,t)dz + \gamma(y)\zeta(t)\right)dy\notag\\
  =&\beta(x,t)+ \int_0^x P(x,y) \beta(y,t)dy+ \Gamma(x)\zeta(t)\label{eq:inv}
 \end{align}
 where
 \begin{align}
 P(x,y)&=\int_y^x {\Phi}(x,z) h(z,y)dz-h(x,y)-{\Phi}(x,y),\\
 \Gamma(x)&=\gamma(x)- \int_0^x {\Phi}(x,y)\gamma(y)dy.
 \end{align}
Applying  the Cauchy-Schwarz inequality, we obtain
\begin{align}
u (0,t)^2&\le m_1(\beta(0,t)^2+\zeta(t)^2),\label{eq:m1}\\
u (1,t)^2&\le m_2(\beta(1,t)^2+\zeta(t)^2+\|\beta[t]\|^2),\label{eq:m2}\\
\|u[t]\|^2&\le m_3 (\zeta(t)^2+\|\beta[t]\|^2),\label{eq:m3}
\end{align}
where
\begin{align}
m_1=&2\max\{1,\max_{\vartheta\in\Theta}\{\Gamma(0;\vartheta)^2\}>0,\notag\\
m_2=&3\max_{\vartheta\in\Theta}\bigg\{1,\int_0^1P(1,y;\vartheta)^2dy,\Gamma(1;\vartheta)^2\bigg\}>0,\notag\\
m_3=&2\max_{\vartheta\in\Theta}\bigg\{\left(1+\left(\int_0^1\int_0^x P(x,y;\vartheta)^2dydx\right)^{\frac{1}{2}}\right)^2,\notag\\
&\int_0^1 \Gamma(x;\vartheta)^2dx\bigg\}>0.
\end{align}
From Poincare inequality, we have that
\begin{align}
-\|\beta_x[t]\|^2\le \frac{1}{2} \beta(1,t)^2-\frac{1}{4}\|\beta[t]\|^2.\label{eq:ine1}
\end{align}
From Agmon's and Young's inequalities, we have that
\begin{align}
\beta(0,t)^2\le \beta(1,t)^2+\|\beta[t]\|^2+\|\beta_x[t]\|^2.\label{eq:ine2}
\end{align}
Applying Young's inequality and the Cauchy-Schwarz inequality into \eqref{eq:Veq}, with using \eqref{eq:m1}--\eqref{eq:m3}, \eqref{eq:ine1}, \eqref{eq:ine2}, we have that
 \begin{align}
\dot V(t) \le& -rr_a\varepsilon \beta(1,t)^2-r_a\varepsilon \beta(1,t)p(t)-r_a\varepsilon \beta(1,t)d(t)\notag\\
&-\frac{1}{8}r_a\varepsilon\int_0^1 \beta(x,t)^2dx+\frac{1}{4}r_a\varepsilon\beta(1,t)^2\notag\\&-\frac{1}{2}r_a\varepsilon\int_0^1 \beta_{x}(x,t)^2dx-\frac{3}{4}a_{\rm m}r_c\zeta(t)^2+\frac{r_c}{a_{\rm m}}b^2\beta(0,t)^2\notag\\
 &+ \eta m(t)-\lambda_d d(t)^2+{\kappa _1 m_2}{\beta}{(1,t)^2} +{\kappa _2}m_1{\beta(0,t)^2}\notag\\
 &+({\kappa _3}m_3+{\kappa _1}m_2)\|\beta(\cdot,t)\|^2\notag\\&+({\kappa _4}+{\kappa _2}m_1+{\kappa _1 m_2}+{\kappa _3}m_3)\zeta(t)^2\notag\\
\le & -\bigg[\left(r-\frac{1}{4}\right)r_a\varepsilon-\frac{r_a\varepsilon}{4r_1}-\frac{r_a\varepsilon}{4r_2}-\frac{r_c}{a_{\rm m}}b^2\notag\\&-{\kappa _1}m_2-\kappa_2 m_1 \bigg]\beta(1,t)^2\notag\\
 &+\eta m(t)-(\lambda_d-r_1r_a\varepsilon) d(t)^2+r_2r_a\varepsilon p(t)^2\notag\\
 &-\left(\frac{3r_c}{4}a_{\rm m}-{\kappa _1 m_2}-{\kappa _2}m_1-{\kappa _3}m_3-{\kappa _4}\right)\zeta(t)^2\notag\\&-\left(\frac{1}{2}r_a\varepsilon-\frac{r_c}{a_{\rm m}}b^2-{\kappa _2}m_1\right)\|\beta_x[t]\|^2\notag\\&-\left(\frac{1}{8}r_a\varepsilon-\frac{r_c}{a_{\rm m}}b^2-{\kappa _1}m_2-{\kappa _2}m_1-{\kappa _3}m_3\right)\|\beta[t]\|^2\label{eq:dV1}
 \end{align}
for $t\in(t_i,t_{i+1})$.
Choosing
\begin{align}
\min\{r_1,r_2\}\ge& \frac{1}{q-\frac{\overline\lambda}{2\varepsilon}-\frac{1}{4}}\ge \frac{1}{r-\frac{1}{4}},\\
r_c>&\frac{8({\kappa _1 m_2}+{\kappa _2}m_1+{\kappa _3}m_3+{\kappa _4})}{3 a_{\rm m}},\\
r_a\ge&\max\bigg\{\frac{2(\frac{r_c}{a_{m}}b^2+{\kappa _1}m_2+{\kappa _2}m_1)}{{(q-\frac{\overline\lambda}{2\varepsilon}-\frac{1}{4})}\varepsilon},\notag\\& \frac{2}{\varepsilon}\left(\frac{r_c}{a_{\rm m}}b^2+{\kappa _2}m_1\right),\notag\\& \frac{16(\frac{r_c}{a_{\rm m}}b^2+{\kappa _1}m_2+{\kappa _2}m_1+{\kappa _3}m_3)}{\varepsilon}\bigg\},\\
\lambda_d\ge& r_1r_a\varepsilon,
\end{align}
where $r=q-\frac{\lambda}{2\varepsilon}$ in \eqref{eq:r} and Assumption \ref{eq:As1} which  ensure $r\ge q-\frac{\overline\lambda}{2\varepsilon}>\frac{1}{4}$ are recalled,
we obtain
\begin{align}
\dot V&\le -\frac{1}{16}r_a\varepsilon \|\beta[t]\|^2-\frac{3r_c}{8}a_{\rm m}\zeta(t)^2\notag\\&\quad+\eta m(t)+r_2r_a\varepsilon p(t)^2.\label{eq:dVa3}
\end{align}
That is,
\begin{align}
\dot V\le-\sigma V(t)+r_2r_a\varepsilon p(t)^2\label{eq:dVa1}
\end{align}
for $t\in(t_i,t_{i+1})$, $i\in \mathbb Z^+$,
where
\begin{align}
\sigma=\min\left\{\frac{1}{8}\varepsilon,\frac{3}{4}a_{\rm m},\eta\right\}.
\end{align}
\begin{clam}\label{cl1}
After $t=t_2$, $p(t)$ defined in \eqref{eq:p} is identically zero, i.e.,
\begin{align}
p(t)\equiv 0,~~t\in[t_{2},\infty).\label{eq:y0}
\end{align}
\end{clam}
\begin{proof}
1) If $u[0]$ and $\zeta(0)$ are zero, it follows that $u[t]$ and $\zeta(t)$ are identically zero for $t\in[0,\infty)$ considering \eqref{eq:ode1}--\eqref{eq:ode4}, \eqref{eq:dU}, \eqref{eq:(1)}. Thus \eqref{eq:y0} holds.

2) Others: Recalling \eqref{eq:p} and Lemma \ref{lem:convergence}, we obtain \eqref{eq:y0}.

The proof of Claim \ref{cl1} is complete.
\end{proof}
Multiplying both sides of \eqref{eq:dVa1} by $e^{\sigma t}$, integrating both sides of \eqref{eq:dVa1} from $t_i$ to $t_{i+1}$, $i\ge 2$, considering Claim \ref{cl1}, we obtain that
 \begin{align}
V(t)\le& V(t_{i})e^{-\sigma(t-t_{i})},~~t\in(t_i,t_{i+1}),~~i\ge 2.\label{eq:dVa}
\end{align}
Recalling Corollary \ref{col}, we know that $V(t)$ defined in \eqref{eq:Va} is continuous. We then have $V(t_{i+1}^-)=V(t_{i+1})$ and $V(t_{i}^+)=V(t_{i})$, and thus we can replace $(t_i,t_{i+1})$ by $[t_i,t_{i+1}]$ in \eqref{eq:dVa}, yielding
 \begin{align}
V(t_{i+1})\le& V(t_{i})e^{-\sigma(t_{i+1}-t_{i})},\label{eq:dVa0}
\end{align}
for $i\ge 2$.

Hence, applying \eqref{eq:dVa0} repeatedly, we obtain from \eqref{eq:dVa} that
\begin{align}
V(t)\le& V(t_{2})\prod_{c=2}^{i-1}e^{-\sigma(t_{c+1}-t_{c})}e^{-\sigma(t-t_{i})}\notag\\
=&V(t_{2})e^{-\sigma(t-t_{2})}\label{eq:dVa2}
\end{align}
for any  $t\in[t_i, t_{i+1}], i\ge 3$. Together with \eqref{eq:dVa} holding for $t\in [t_{2},t_{3}]$, we  obtain
\begin{align}
V(t)\le V(t_{2})e^{-\sigma(t-t_{2})}\label{eq:Vt}
\end{align}
for $t\ge t_2$.

In the following claim, we analyze the responses on $t\in[0,t_2]$.
\begin{clam}\label{cl3}
For the finite time $t_2$ in Claim \ref{cl1}, the following estimate holds,
\begin{align}
V(t)\le V(0)e^{\bar\sigma t},~~t\in[0,t_2]\label{eq:vin}
\end{align}
where
\begin{align}
\bar\sigma=&\frac{1}{\xi_3}\max\{r_2r_a\varepsilon\Upsilon_p,1\}\notag\\&-\min\left\{\frac{1}{8}\varepsilon,\frac{3}{4}a_{\rm m},\eta+1\right\}\label{eq:barsigma}
\end{align}
for some positive $\Upsilon_p$.
\end{clam}
\begin{proof}
Bounding $p(t)^2$ defined in \eqref{eq:p} on $t\in[0,t_{2}]$ as
\begin{align}
p(t)^2\le \bar\Upsilon_p(\|u[t]\|^2+\zeta(t)^2).
\end{align}
where
\begin{align}
\bar\Upsilon_p=&2\max\bigg\{\max_{\vartheta\in\Theta}\left\{\int_0^1 (K_1(1,y;\theta)-K_1(1,y;\vartheta))^2dy\right\},\notag\\
&\max_{\vartheta\in\Theta}\{(K_2 (1;\theta)-K_2 (1;\vartheta))^2\}\bigg\}.
\end{align}
Recalling \eqref{eq:m3}, we obtain
\begin{align}
p(t)^2\le \Upsilon_p (\|\beta[t]\|^2+\zeta(t)^2),~~t\in[0,t_{2}]\label{eq:p2}
\end{align}
where the positive constant $\Upsilon_p$ is
\begin{align}
\Upsilon_p=\max\left\{\bar\Upsilon_p m_3, \bar\Upsilon_p (m_3+1)\right\}.
\end{align}

For $0\le i< 2$, we have from \eqref{eq:dVa3} and \eqref{eq:p2} that
\begin{align}
\dot V(t)&\le -\frac{1}{16}r_a\varepsilon \|\beta[t]\|^2-\frac{3r_c}{8}a_{\rm m}\zeta(t)^2+\eta m(t)+m(t)\notag\\&\quad-m(t)+r_2r_a\varepsilon\Upsilon_p (\|\beta[t]\|^2+\zeta(t)^2)\notag\\
&\le \bar\sigma V(t),~~~t\in(t_i,t_{i+1})
\end{align}
where $\bar\sigma$ is given in \eqref{eq:barsigma}.

We then have
\begin{align}
V(t)\le V(t_i)e^{\bar\sigma(t-t_{i})},\label{eq:Vlf}
\end{align}
for $t\in(t_i,t_{i+1}), 0\le i< 2$.
Recalling again the continuity of $V(t)$,  and thus we can replace $(t_i,t_{i+1})$ by $[t_i,t_{i+1}]$ in \eqref{eq:Vlf}.

Then we have
\begin{align}
V(t_{1})\le V(0)e^{\bar\sigma t_{1}}.\label{eq:vti+1}
\end{align}
Recalling \eqref{eq:Vlf} and applying \eqref{eq:vti+1}, we have that
\begin{align}
V(t)&\le V(0)e^{{\bar\sigma}t_{1}}e^{{\bar\sigma}(t-t_{1})}\notag\\
&=V(0)e^{{\bar\sigma}t}
\end{align}
for any $t\in[t_1,t_{2}]$. It is obtained from \eqref{eq:Vlf} that $V(t)\le V(0)e^{{\bar\sigma}t}$ also holds for $t\in[0,t_1]$.  Therefore, \eqref{eq:vin} holds.

The proof of Claim \ref{cl3} is complete.
\end{proof}

We obtain from Claim \ref{cl3} that
\begin{align}
V(t_2)\le V(0)e^{{\bar\sigma}t_2}.\label{eq:V0}
\end{align}
By virtue of \eqref{eq:Vt}, \eqref{eq:vin}, \eqref{eq:V0}, we have
\begin{align}
V(t)\le& V(0)e^{(\bar\sigma+{\max\{\sigma,-\bar\sigma\}) t_{2}}}e^{-\sigma t}\notag\\
\le&  V(0)e^{2T(\bar\sigma+{\max\{\sigma,-\bar\sigma\})}}e^{-\sigma t}
\end{align}
for $t\in[0,\infty)$, where
\begin{align}
t_{2}\le 2T\label{eq:t2}
\end{align}
ensured by \eqref{eq:tk1} is recalled.

Recalling \eqref{eq:barOmega}, we have
 \begin{align}
\bar\Omega(t)\le \Upsilon_1\bar\Omega(0)e^{-\sigma t}, t\ge 0\label{13eq:Vnorm1}
\end{align}
where the positive constant $\Upsilon_1$ is
\begin{align}
\Upsilon_1=\frac{\xi_4}{\xi_3}e^{2T(\bar\sigma+{\max\{\sigma,-\bar\sigma\})}}.\label{eq:Upsilon}
\end{align}
From \eqref{eq:theta}, \eqref{eq:hattheta} and Lemma \ref{lem:convergence},  we know that
\begin{align}
\tilde\theta(t)=0,~~~\forall t\ge t_2,
\end{align}
and
\begin{align}
\left|\tilde\theta(t)\right|\le \left|\tilde\theta(0)\right|, ~~~\forall t\in[0, t_2],
\end{align}
which is ensured by Lemma \ref{13lem:theta}.

Recalling \eqref{eq:t2}, the following estimate holds
\begin{align}
\left|\tilde\theta(t)\right|\le e^{2\sigma T}  \left|\tilde\theta(0)\right| e^{-\sigma t}.
\end{align}
Therefore, together with \eqref{13eq:Vnorm1}, we have that
\begin{align}
\left(\bar\Omega(t)+\left|\tilde\theta(t)\right|\right)\le \Upsilon\left(\bar\Omega(0)+\left|\tilde\theta(0)\right|\right)e^{-\sigma t}\label{eq:barom}
\end{align}
where
\begin{align}
\Upsilon=\max\{\Upsilon_1,e^{2\sigma T}\}.
\end{align}
Applying the invertibility of the transformations \eqref{eq:contran1b}, \eqref{eq:contran2b}, \eqref{eq:thirdtran},
we thus obtain \eqref{13eq:th1norm1}.

2) Now we prove the second of the two portions of the theorem. It follows from \eqref{eq:ode1}--\eqref{eq:ode4}, \eqref{eq:dU}, \eqref{eq:(1)} and $u[0]=0$, $\zeta(0)=0$ that $u[t]\equiv 0$ and $\zeta(t)\equiv 0$ for $t\in[0,\infty)$. Recalling \eqref{eq:d}, \eqref{eq:dm}, we know $m(t)=m(0)e^{-\eta t}$, i.e., $|m(t)|\le|m(0)|$, for $t\in[0,\infty)$. Also, it is obtained (by the step
method) from Lemma \ref{13lem:theta}  that $\hat\theta(t)\equiv \hat\theta(0)$ for $t\in[0,\infty)$, i.e., $\tilde\theta(t)\equiv \theta-\hat\theta(0)$ for $t\in[0,\infty)$. Therefore, \eqref{eq:bound} is obtained.

The proof of the theorem is complete.
\end{proof}
In the proposed control system, all conditions of the design parameters are cascaded rather than coupled, and only depend on the known parameters. An order of selecting the design parameters is shown in Figure \ref{fig:0}.
\begin{figure}
\centering
\includegraphics[width=9cm]{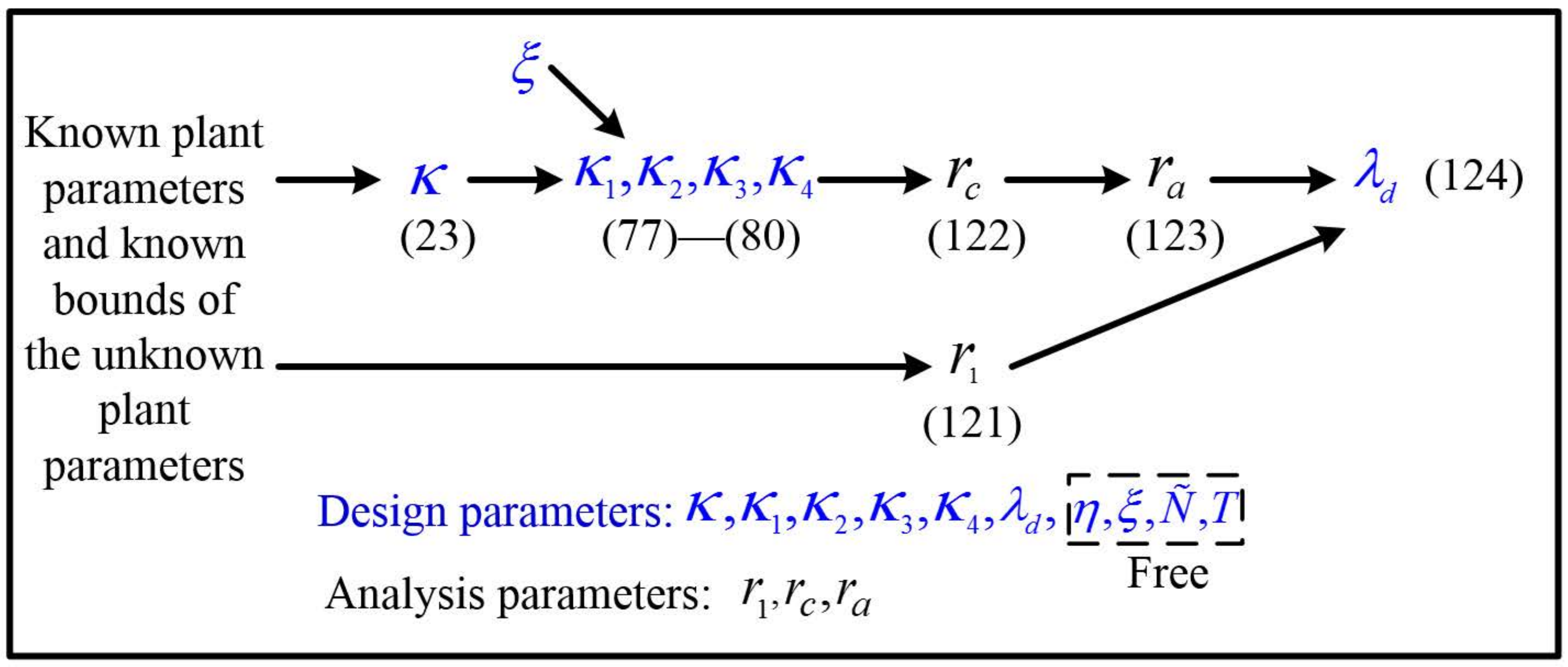}
\caption{An order of selecting the design parameters.}
\label{fig:0}
\end{figure}
\section{Simulation}\label{13sec:sim}
\subsection{Model}
The simulation model is \eqref{eq:ode1}--\eqref{eq:ode5} with the following parameters $$a=1.5,~b=1,~\varepsilon=1,~\lambda=3,~q=5.$$
The bounds $\overline \lambda, \underline\lambda, \overline a, \underline a$ of the unknown parameters $\lambda$, $a$ are set as $0, 5, 0, 3$, respectively.  Initial conditions are defined as
\begin{align}
u(x, 0) =x^2\sin(2\pi x),~~\zeta(0)=5.\label{eq:IC}
\end{align}
In the numerical calculation by the finite difference method, the model is discretized with the time step of $0.004$ and the space step of $0.05$.
\subsection{Design Parameters}
The design parameters are chosen as $\xi=1.1$, $\tilde N=5$, $T=1.2$, $\eta=15$, $\kappa=16$, $\kappa_1=\kappa_2=\kappa_3=\kappa_4=100$, $\lambda_d=20$, and $n$ in \eqref{eq:adaptivelaw} is truncated at $15$ and the initial condition of $m(t)$ is set as $m(0)=-500$. The function of free design parameters $\tilde N$, $T$, $\eta$, $\xi$ and $m(0)$ in adjusting the response of the closed-loop system  is illustrated as follows. A larger $m(0)$ can reduce the triggering times at the initial stage, which allows more data being collected for least-squares parameter identification. Even though a larger overshoot of the plant norms may appear due to the large $m(0)$, the increase of $\eta$ can fasten the convergence of $m(t)$ to zero, together with the decrease of $\xi$, which can make the plant states resampled more frequently, especially after the parameter estimates reaching the true values (which can always be achieved in initial several updates), and thus increase the decay rate of the plant states. Besides, as mentioned when introduce the design parameters $\tilde N, T$, the increase of $\tilde N$ allows the data in more time intervals  to be used in parameter identification, which can improve the accuracy and robustness of the identifier, and $T$ is chosen to avoid less frequent updates of parameter estimates considering the operation time is only $4$ s.
\subsection{Gain Kernels}
The kernels $\lambda (x)$, ${\Psi}(x,y)$ are directly obtained  from \eqref{eq:ga}, \eqref{eq:Psi} (using the modified Bessel function given in (A.10) in \cite{krstic2008} with $n=1$ and cutting off $m$ in (A.10) at 15), where the unknown coefficients are replaced by the piecewise-constant estimates. The approximate solution  $h(x,y)$ of \eqref{eq:h1}--\eqref{eq:h4} where the unknown coefficients are replaced by the piecewise-constant estimates is obtained  by the finite difference
method on a lower
triangular domain discretized as a grid
with the uniformed interval of $0.05$ (the
spatial variables $x$ and $y$ were discretized using 21 grid points
each). The value at each grid point is denoted as $h_{i,j}$, $1\le j\le i\le 21$, $i,j\in \mathbb N$. According to \eqref{eq:h4}, we know $h_{1,1}=0$. Together with \eqref{eq:h3}, we have $h_{i,i}=0$, $i=1,2,\cdots, 21$. Then $h_{2,1}$ can be solved via \eqref{eq:h1}. For representing
the two-order derivatives in \eqref{eq:h2} by the finite difference scheme,  we adopt the following approximate $h_{i,i-1}=h_{2,1}, i=2,\cdots,21$. The kernel $h$ will be recomputed when the parameter estimates are changed in the evolution. In the simulation results which will be shown later, we know that $h$ is recomputed twice, according to the parameter estimates $\hat\lambda$ and $\hat a$.
\subsection{Simulation Results}
The open-loop response of the ODE state  $\zeta(t)$ and PDE state $u(x,t)$ are shown in Figures \ref{fig:2}, \ref{fig:3}, from which we observe that the plant is unstable. Applying the proposed adaptive event-triggered controller $U_d$ defined in \eqref{eq:dU},
it is shown in Figures \ref{fig:4}, \ref{fig:5} that the ODE state $\zeta(t)$ and PDE state $u(x,t)$ are convergent to zero.
The piecewise-constant control input $U_d(t)$ defined in \eqref{eq:dU} and the continuous-in-state control signal $U_c(t)$ \eqref{eq:cU} used in ETM are shown in Figure \ref{fig:6}.  For the control input $U_d(t)$,  the estimate $\hat\theta$ is recomputed and the states $u$, $\zeta$ are resampled simultaneously, the total number of triggering times is $26$, the minimal dwell-time is $0.0404$ s, which is much larger than the highly conservative minimal dwell time estimate (whose order of magnitude is $10^{-6}$ s) obtained from \eqref{eq:undertaua}, \eqref{eq:undertau} in Lemma \ref{lem:dwell}. There are two "jumps" in the continuous-in-state control signal $U_c(t)$ \eqref{eq:cU} at the first two triggering times, because of the updates in the parameter estimates which are shown in Figure \ref{fig:7}, where the estimates reach the true values after two triggering times (the exact estimates are not obtained at the time of the first event under the nonzero initial condition \eqref{eq:IC}  as Lemmas \ref{13lem:theta}, \ref{lem:keep} imply, because of the approximation adopted in the simulation, including the discretization of time and space, and truncation of $n=1,2,\cdots$ in the estimator \eqref{eq:adaptivelaw}).
\begin{figure}
\centering
\includegraphics[width=9.5cm]{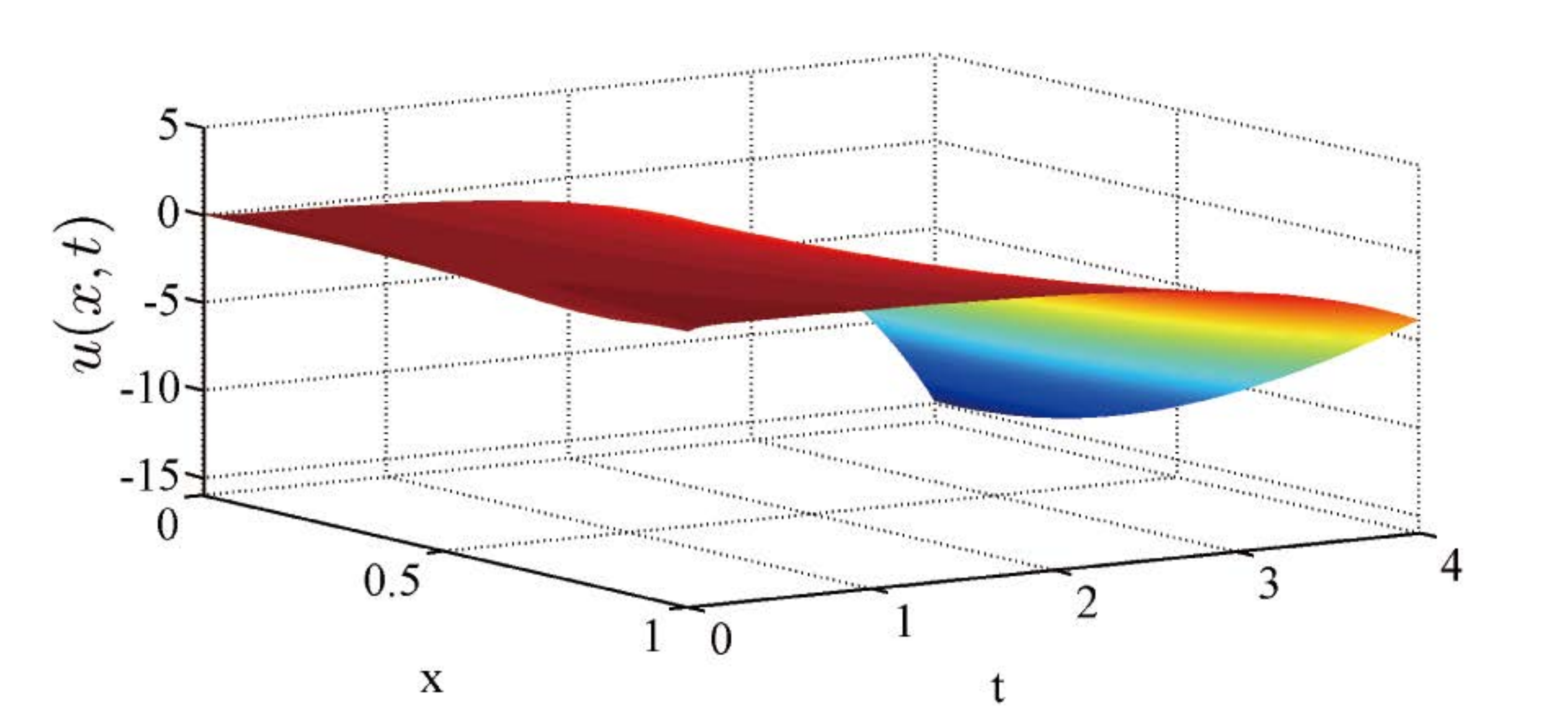}
\caption{The evolution of $u(x,t)$ in open loop.}
\label{fig:2}
\end{figure}
\begin{figure}
\centering
\includegraphics[width=9.5cm]{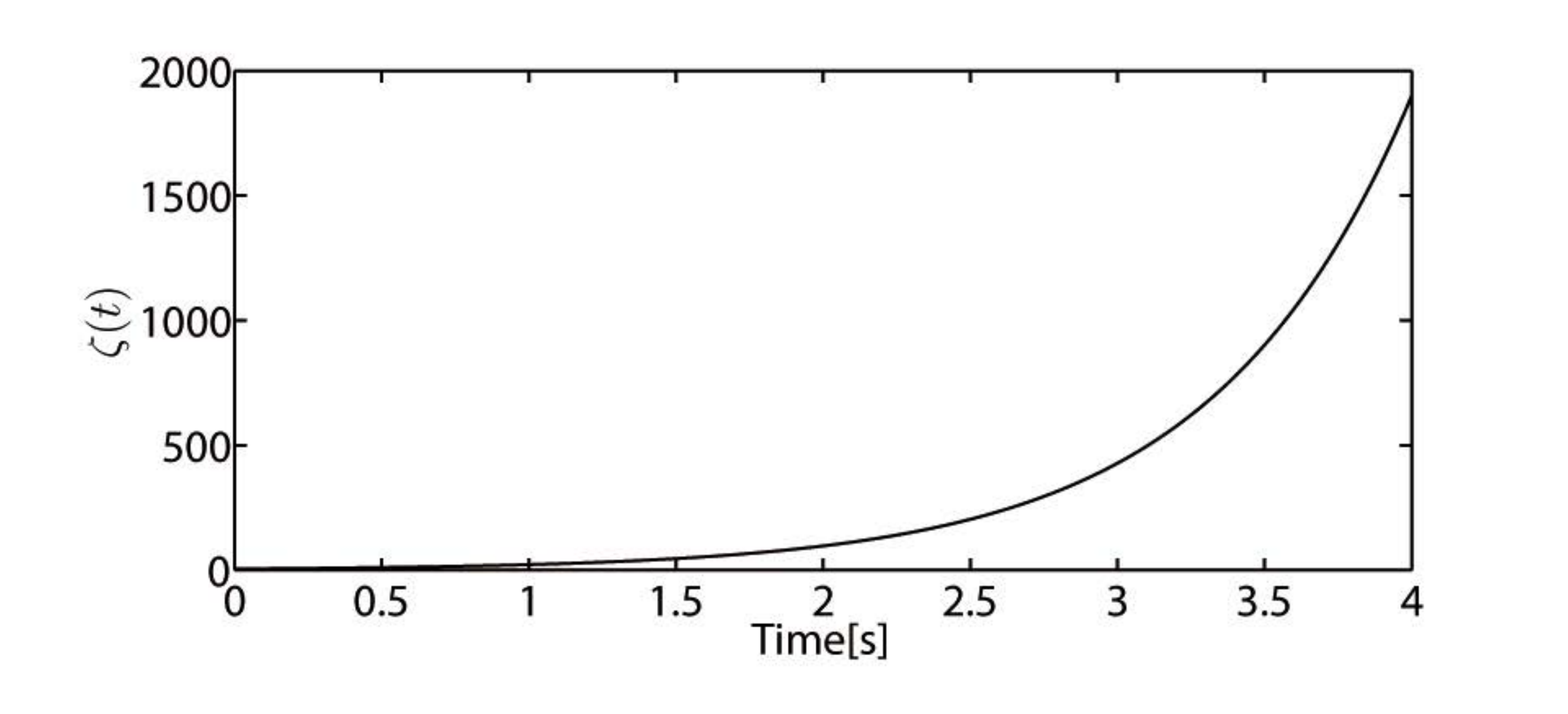}
\caption{The evolution of $\zeta(t)$ in open loop.}
\label{fig:3}
\end{figure}
\begin{figure}
\centering
\includegraphics[width=9.5cm]{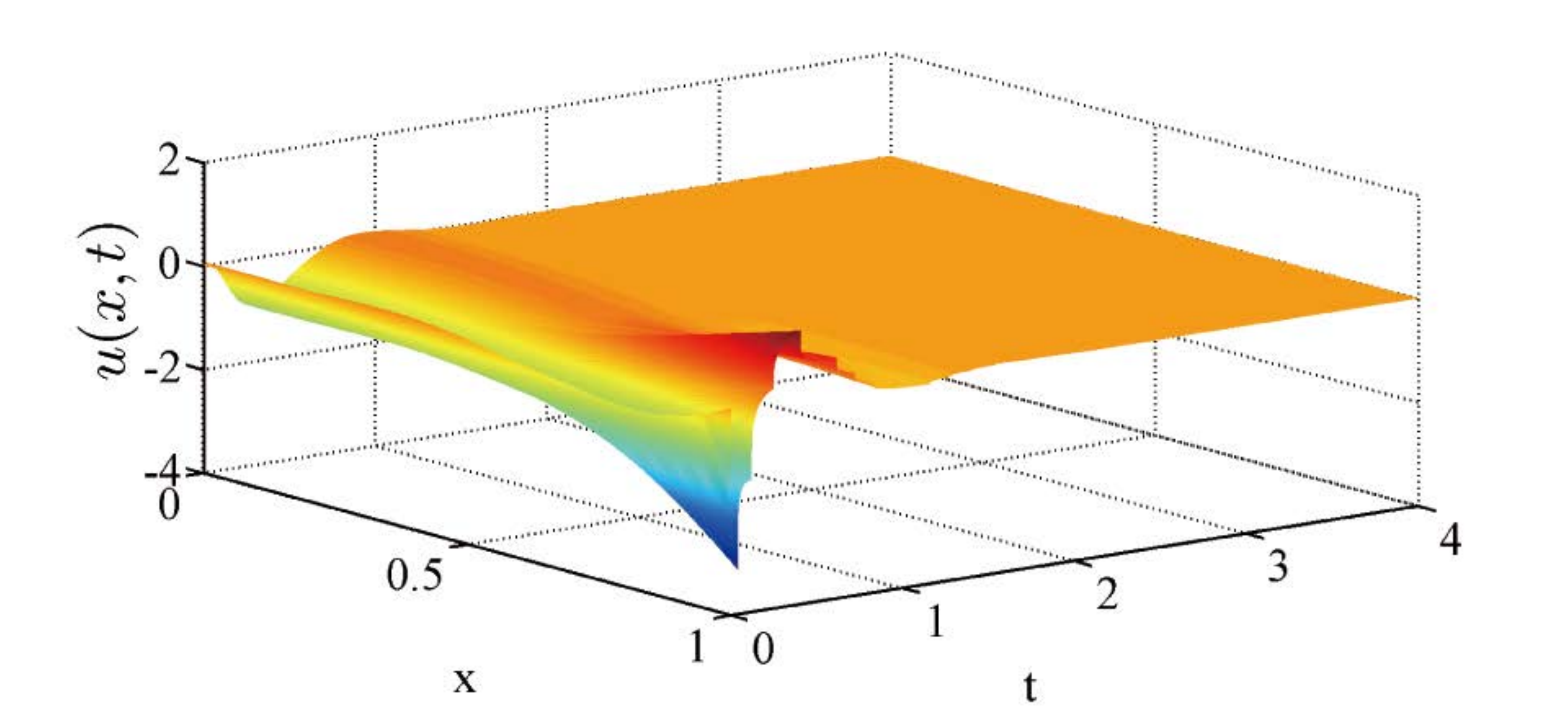}
\caption{The evolution of $u(x,t)$ under the control input $U_d(t)$ defined in \eqref{eq:dU}.}
\label{fig:4}
\end{figure}
\begin{figure}
\centering
\includegraphics[width=9.5cm]{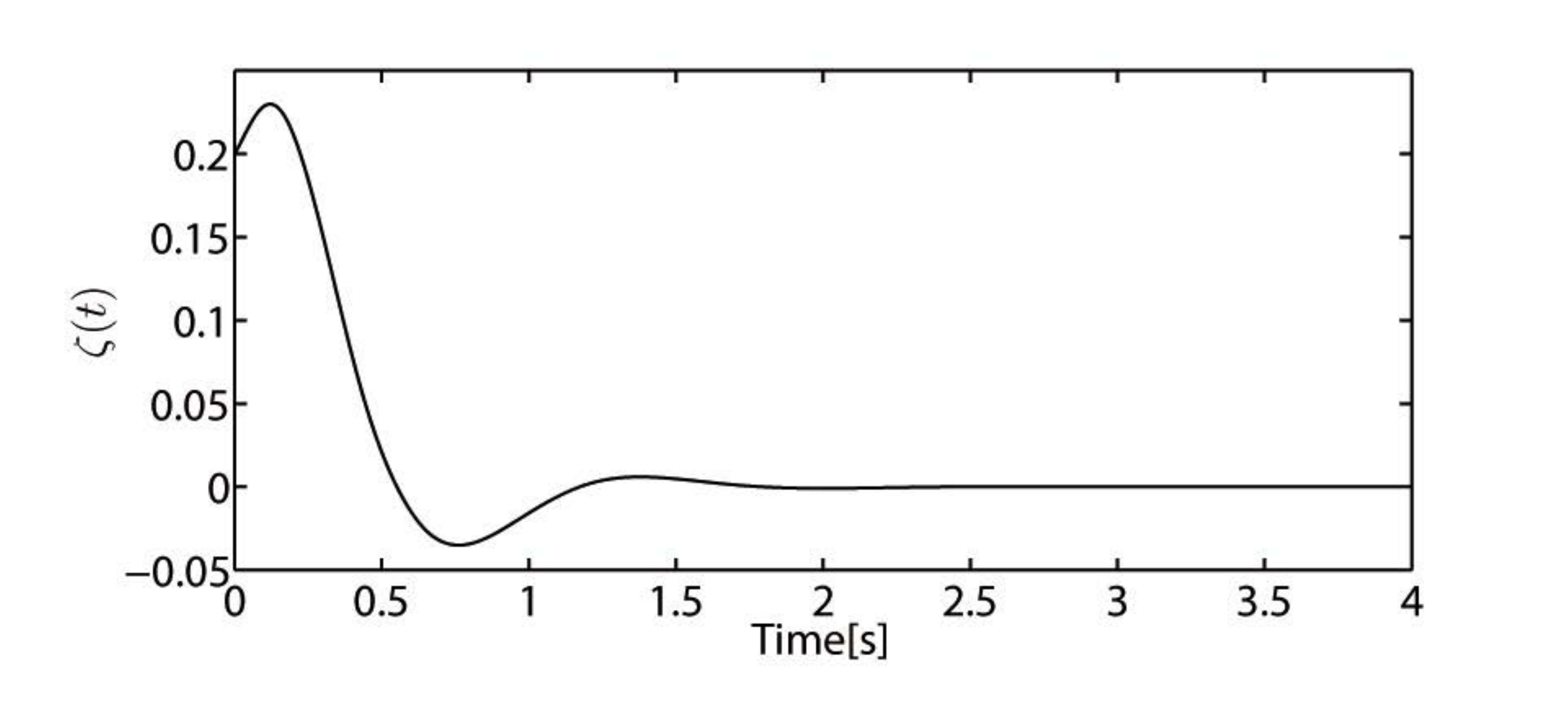}
\caption{The evolution of $\zeta(t)$ under the control input $U_d(t)$ defined in \eqref{eq:dU}.}
\label{fig:5}
\end{figure}
\begin{figure}
\centering
\includegraphics[width=9.5cm]{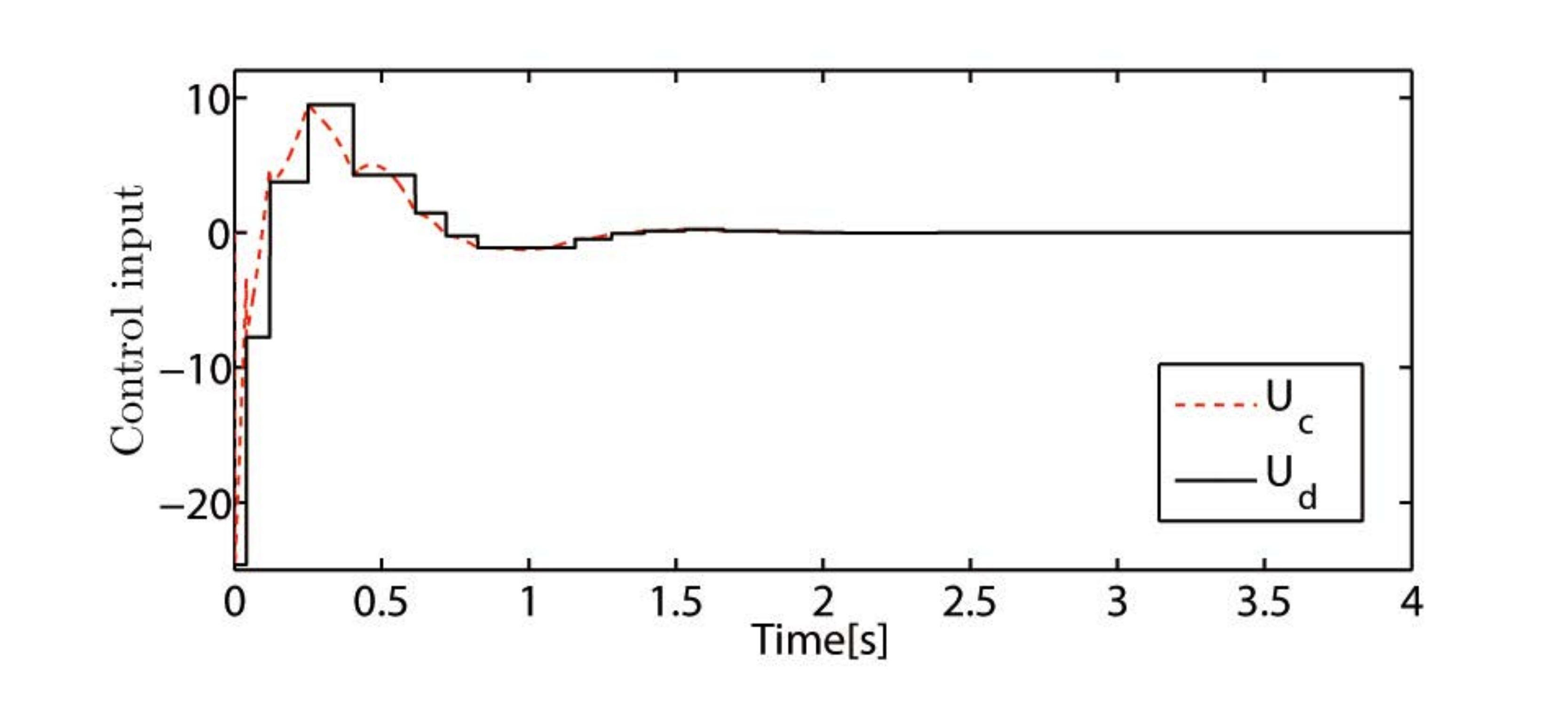}
\caption{The piecewise-constant control input $U_d(t)$ \eqref{eq:dU} and the continuous-in-state control signal $U_c(t)$ \eqref{eq:cU} used in ETM.}
\label{fig:6}
\end{figure}
\begin{figure}
\centering
\includegraphics[width=9.5cm]{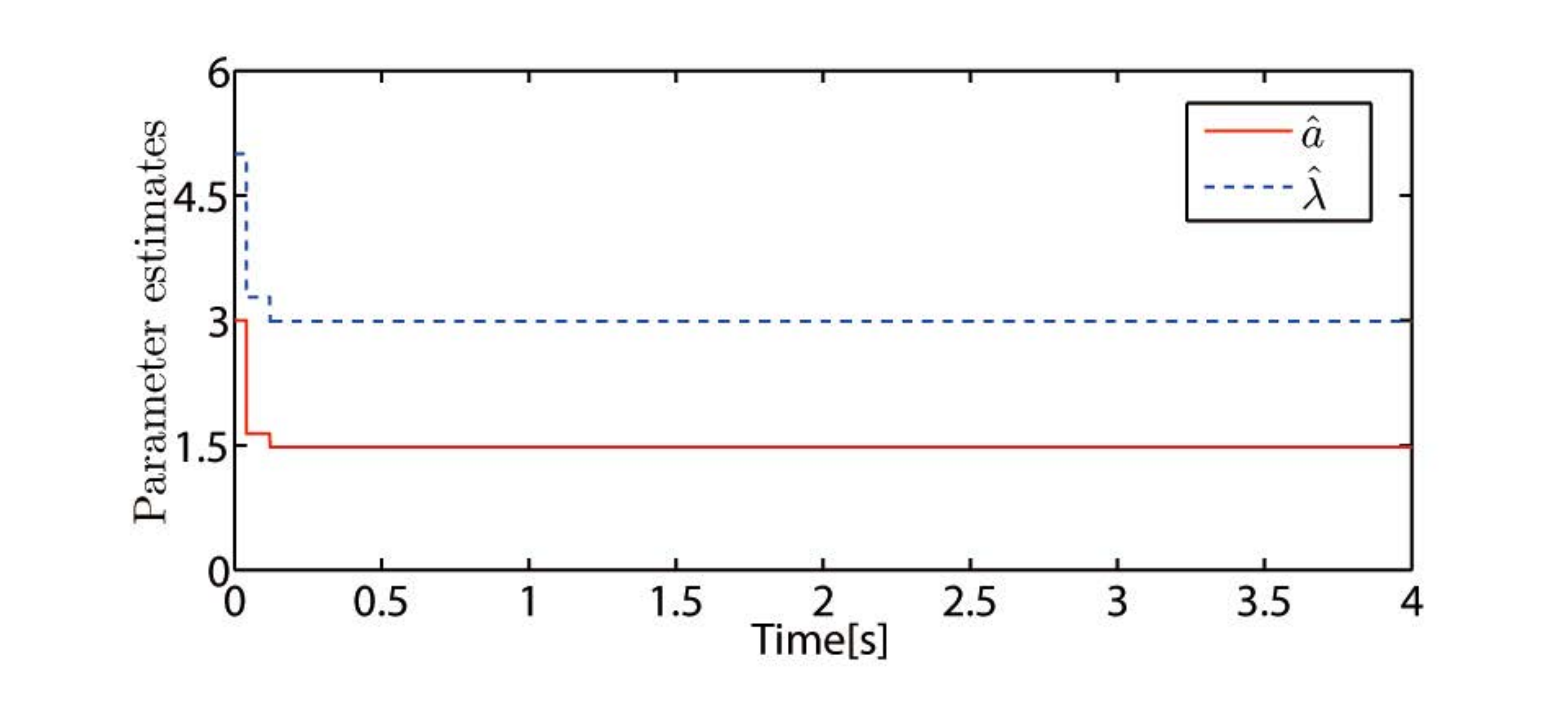}
\caption{The evolution of the parameter estimates.}
\label{fig:7}
\end{figure}
\begin{figure}
\centering
\includegraphics[width=9.5cm]{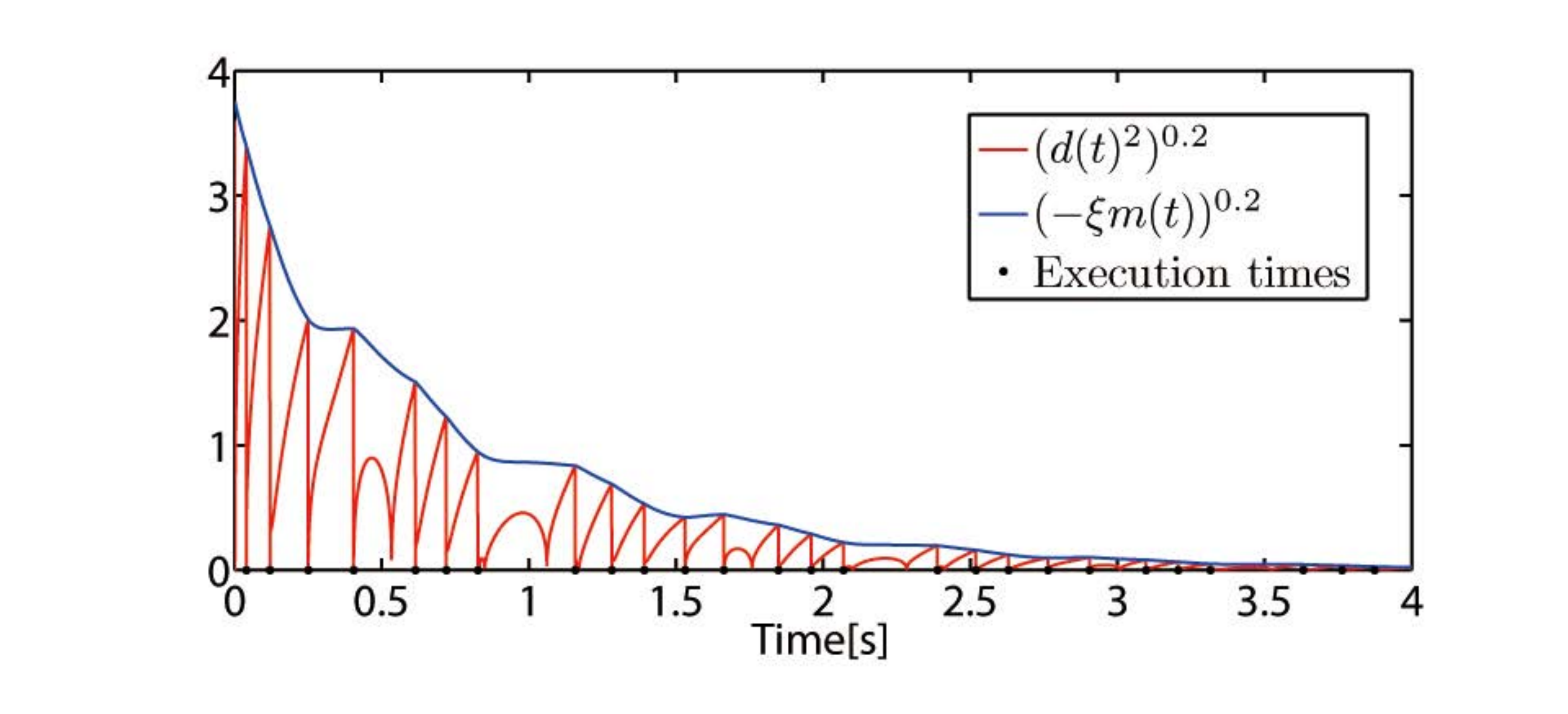}
\caption{The evolution of $(d(t)^2)^{0.2}$ and $(-\xi m(t))^{0.2}$.}
\label{fig:8}
\end{figure}
\begin{figure}
\centering
\includegraphics[width=9.5cm]{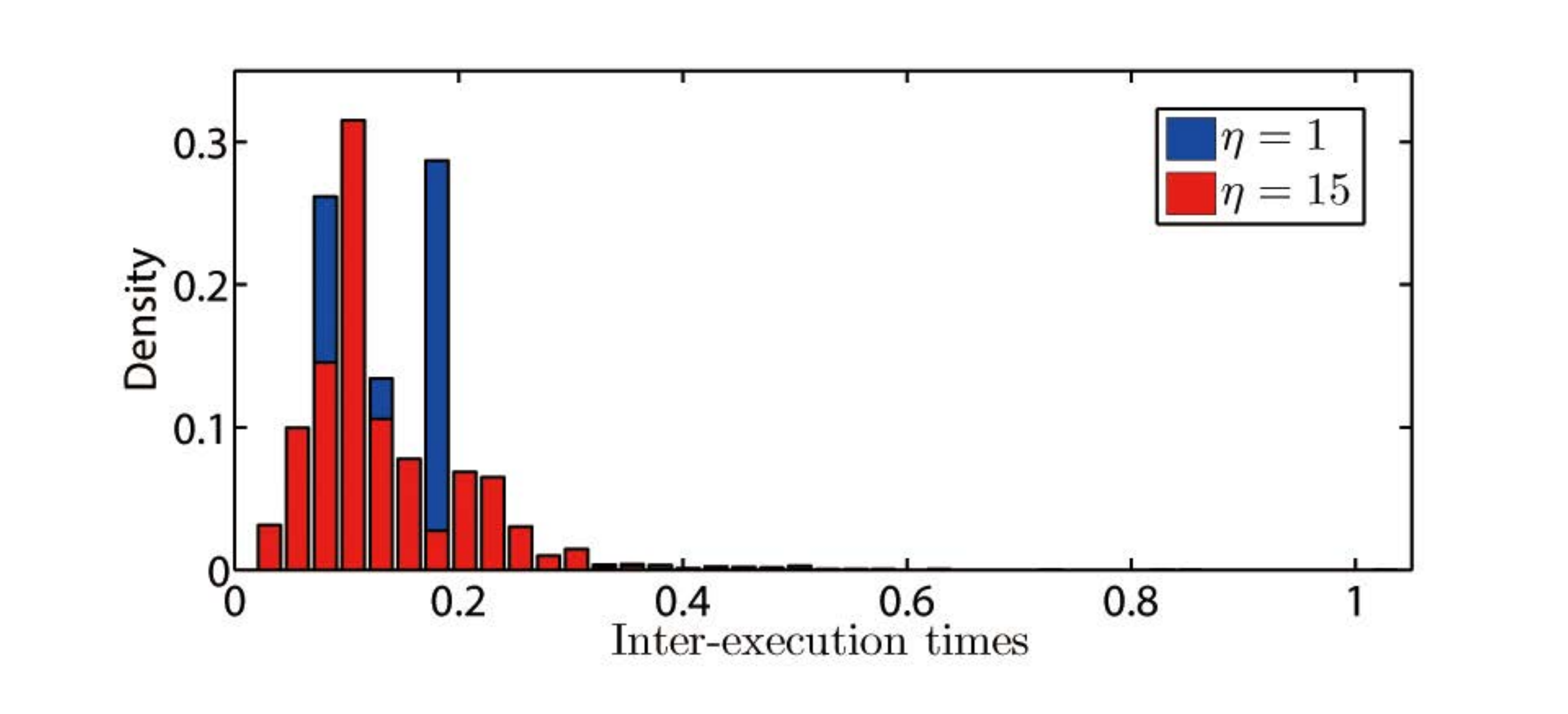}
\caption{Density of the inter-execution times computed for 100
different initial conditions given by $u(x,0)=x^2\sin(\bar n \pi x)$, $\zeta(0)=0.2$, $\bar n=1,2,\ldots, 100$.}
\label{fig:9}
\end{figure}

Figure \ref{fig:8} shows the time evolution of the functions
in the triggering condition \eqref{eq:tk1} and the execution times, {where an event is
generated, the control value is updated and $d(t)$ is reset to zero, when the trajectory
$d(t)^2$ reaches the trajectory $ - \xi m(t)$.}

 {Finally, we run simulations for 100 different initial conditions and compute the inter-execution times
between two triggering times. The density of the inter-execution times is shown in Figure \ref{fig:9}, from which we know that the prominent inter-execution times are around $0.1$ s when $\eta=15$, and increase to around 0.2 s when $\eta$ decreases  to 1}.
\section{Conclusion and future work}\label{sec:conclusion}
In this paper, we have proposed an adaptive event-triggered boundary
control scheme for a parabolic PDE-ODE system, where the reaction coefficient of the parabolic PDE, and the system parameter of the ODE
are unknown, and both of the parameter estimates and
control input employ piecewise-constant values.  The controller includes an event-triggering mechanism to determine the
{synchronous} update times of both the batch least-squares
identifier and plant states
in the control law. We have proved that the proposed control
guarantees: 1) no Zeno phenomenon occurs; 2) parameter estimates
are convergent to the true values in finite time under most initial conditions of the plant (all initial conditions except a set of measure zero); 3) the plant states are exponentially regulated
to zero. The effectiveness of the proposed
design is verified by a numerical example. In the future
work, the state-feedback control design will be extended to
the output-feedback type conforming to available sensors in
practice.
\section*{Appendix}
\subsection{Calculating Conditions of $\gamma(x)$}\label{ap:Ap1}
\setcounter{equation}{0}
\renewcommand{\theequation}{A.\arabic{equation}}
Inserting \eqref{eq:contran2b} into \eqref{eq:Ib1}, recalling \eqref{eq:I1}, \eqref{eq:am}, we have
\begin{align}
&\dot \zeta (t) +a_{\rm m}\zeta (t) - {b}w(0,t) + b\gamma(0)\zeta(t)\notag\\
=&\dot \zeta (t)-a \zeta (t)- {b}w(0,t)+(b\kappa+b\gamma(0))\zeta (t)\notag\\
=&(b\kappa+b\gamma(0))\zeta (t)=0.\label{eq:conga1}
\end{align}
Inserting \eqref{eq:contran2b} into \eqref{eq:Ib2}, recalling \eqref{eq:I2}, \eqref{eq:Ib1}, we have
\begin{align}
&{v _t}(x,t)-{\varepsilon}{v _{xx}}(x,t)+\gamma(x){b}v (0,t)\notag\\
=&w_t(x,t) - \gamma(x)\dot\zeta(t)-{\varepsilon}w_{xx}(x,t)\notag\\& + {\varepsilon}\gamma''(x)\zeta(t)+\gamma(x){b}v (0,t)\notag\\
=& \gamma(x)a_{\rm m}\zeta (t) -\gamma(x) {b}v (0,t) + {\varepsilon}\gamma''(x)\zeta(t)+\gamma(x){b}v (0,t)\notag\\
=&(\gamma(x)a_{\rm m}+{\varepsilon}\gamma''(x))\zeta (t)=0.\label{eq:conga2}
\end{align}
By virtue of \eqref{eq:Ib3} and \eqref{eq:I3}, we have
\begin{align}
v_x(0,t) = w_x(0,t) - \gamma'(0)\zeta(t)= - \gamma'(0)\zeta(t)=0.\label{eq:conga3}
\end{align}
According to \eqref{eq:conga1}--\eqref{eq:conga3}, the conditions of $\gamma(x)$ are obtained as
\begin{align}
 &{\varepsilon}\gamma''(x)+a_{\rm m}\gamma(x)=0,\label{eq:ga1}\\
 &\gamma(0)=-\kappa,\\
 &\gamma'(0)=0,\label{eq:ga3}
\end{align}
which are satisfied by \eqref{eq:ga}.
\subsection{Calculating Conditions of $h(x,y)$}\label{ap:Ap2}
\setcounter{equation}{0}
\renewcommand{\theequation}{B.\arabic{equation}}
Inserting \eqref{eq:thirdtran} into \eqref{eq:targ2}, using \eqref{eq:Ib2}, \eqref{eq:Ib3}, applying integration by parts twice, we obtain
\begin{align}
&\beta_t(x,t)-\varepsilon\beta_{xx}(x,t)\notag\\
=&v_t(x,t)-\int_0^x h(x,y)v_t(y,t)dy-{\varepsilon}v_{xx}(x,t)\notag\\&+{\varepsilon}\int_0^x h_{xx}(x,y)v(y,t)dy+{\varepsilon}h_{x}(x,x)v(x,t)\notag\\&+{\varepsilon}h_x(x,x)v(x,t)+{\varepsilon}h_y(x,x)v(x,t)+{\varepsilon}h(x,x)v_x(x,t)\notag\\
=&-\gamma(x){b}v (0,t)-\int_0^x h(x,y){\varepsilon}v_{xx}(y,t)dy\notag\\&+{b}\int_0^x h(x,y)\gamma(y)dyv (0,t)+{\varepsilon}\int_0^x h_{xx}(x,y)v(y,t)dy\notag\\
&+2{\varepsilon}h_{x}(x,x)v(x,t)+{\varepsilon}h_y(x,x)v(x,t)+{\varepsilon}h(x,x)v_x(x,t)\notag\\
=&-\gamma(x){b}v (0,t)-h(x,x){\varepsilon}v_{x}(x,t)+h(0,0){\varepsilon}v_{x}(0,t)\notag\\&+h_y(x,x){\varepsilon}v(x,t)-h_y(x,0){\varepsilon}v(0,t)\notag\\&-\int_0^x h_{yy}(x,y){\varepsilon}v(y,t)dy+{b}\int_0^x h(x,y)\gamma(y)dyv (0,t)\notag\\&+{\varepsilon}\int_0^x h_{xx}(x,y)v(y,t)dy+2{\varepsilon}h_{x}(x,x)v(x,t)\notag\\&+{\varepsilon}h_y(x,x)v(x,t)+{\varepsilon}h(x,x)v_x(x,t)\notag\\
=&-\left(h_y(x,0){\varepsilon}+\gamma(x){b}-{b}\int_0^x h(x,y)\gamma(y)dy\right)v (0,t)\notag\\&+({\varepsilon}h(x,x)-h(x,x){\varepsilon})v_{x}(x,t)\notag\\&+2{\varepsilon}(h_y(x,x)+h_x(x,x))v(x,t)\notag\\
&-{\varepsilon}\int_0^x (h_{yy}(x,y)-h_{xx}(x,y))v(y,t)dy=0.\label{eq:conh1}
\end{align}
By virtue of  \eqref{eq:targ3} and \eqref{eq:Ib3}, we have
\begin{align}
\beta_x(0,t)=v_x(0,t)-h(0,0)v(0,t)=-h(0,0)v(0,t)=0.\label{eq:conh2}
\end{align}
According to \eqref{eq:conh1}, \eqref{eq:conh2}, we obtain the conditions \eqref{eq:h1}--\eqref{eq:h4}.

\end{document}